\documentclass[12pt]{article}
\usepackage{amsmath,amssymb,amsthm,mathtools,enumitem, hyperref}
\usepackage{authblk}
\usepackage{color}
\usepackage{geometry}
\usepackage[mathscr]{eucal}
\newtheorem{Theorem}{Theorem}[section]

\newtheorem{Lemma}{Lemma}[section]

\newtheorem{Remark}{Remark}[section]
\newtheorem{Definition}{Definition}[section]
\newtheorem{Corollary}{Corollary}[section]

\begin{document}

\def\eR{\mathbf{R}}
\def\Rd{{\eR}^d}
\def\Rdd{{\eR}^{d\times d}}
\def\Rdsym{{\eR}^{d\times d}_{sym}}
\def\eN{\mathbf{N}}
\def\eZ{\mathbf{Z}}
\def\IdM{\mathbb{I}_d}
\def\CrlS{\mathcal{S}}
\def\tder{\partial_t}
\def\SymG{\mathbb{D}}
\def\SymGDev{\mathbb{D}^{D}}
\def\dd{\mbox{d}}
\newcommand{\essinf}{\operatorname{ess\,inf}}
\newcommand{\esssup}{\operatorname{ess\,sup}}
\newcommand{\supp}{\operatorname{supp}}
\newcommand{\spn}{\operatorname{span}}
\newcommand\dx{\; \mathrm{d}x}
\newcommand\dy{\; \mathrm{d}y}
\newcommand\dz{\; \mathrm{d}z}
\newcommand\dt{\; \mathrm{d}t}
\newcommand\ds{\; \mathrm{d}s}
\newcommand\diff{\mathrm{d}}
\newcommand\dvr{\mathop{\mathrm{div}}\nolimits}
\newcommand\pat{\partial_t}
\newcommand\sym{\mathrm{sym}}
\newcommand\diam{\mathrm{diam}}
\newcommand\tr{\mathop{\mathrm{tr}}\nolimits}
\newcommand\lspan{\mathop{\mathrm{span}}\nolimits}
\newcommand\dist{\mathrm{dist}}
\title{Singular limit for the compressible Navier--Stokes equations with the hard sphere pressure law on expanding domains}
\author{Martin Kalousek, \v{S}\'arka Ne\v{c}asov\'a}
\maketitle

{\bf The article was inspired by several discussions with our colleague and friend Anton\' \i n Novotn\' y. We never forget him.}

\begin{abstract} 
	The article is devoted to the asymptotic limit of the compressible Navier-Stokes system with a pressure obeying a hard--sphere equation of state on a domain expanding to the whole physical space $\eR^3$. Under the assumptions that acoustic waves generated in the case of ill-prepared data do not reach the boundary of the expanding domain in the given time interval and a certain relation between the Reynolds and Mach numbers and the radius of the expanding domain we prove that the target system is the incompressible Euler system on $\eR^3$. We also provide an estimate of the rate of convergence expressed in terms of characteristic numbers and the radius of domains.
\end{abstract}
\section{Introduction}
Let $T>0$ and $\Omega\subset\eR^d$, $d\in \{2,3\}$ be a bounded domain. We consider the compressible Navier-Stokes system in the time-space cylinder $Q_T=(0,T)\times\Omega$
\begin{equation}\label{NSEq}
\begin{alignedat}{2}
\tder \rho+\dvr(\rho u)=\, &0&&\text{ in }Q_T,\\
\tder(\rho u)+\dvr(\rho u\otimes u-\mathbb{S}(\nabla u))+\nabla p(\rho)=\, &\rho f&&\text{ in }Q_T,\\
\rho(0,\cdot)=\rho_0,\ u(0,\cdot)=\, &u_0&&\text{ in }\Omega,\\
u=\, &0&&\text{ on }(0,T)\times\partial\Omega
\end{alignedat}
\end{equation}
for the unkown density $\rho:Q_T\to\eR$ and the velocity $u:Q_T\to \eR^d$. The structural relation between the pressure $p$ and the density $\rho$ is discussed later. The external forces are denoted by $f$. 
Here, $\mathbb S(\nabla u)$ denotes the Newtonian stress tensor defined as
\begin{equation}
\mathbb S(\nabla u)=\mu^S\left(\frac{1}{2}(\nabla u+(\nabla u)^\top)-\frac{1}{d}\dvr u\mathbb I_d\right)+\mu^B\dvr u\mathbb I_d,
\end{equation}
where $\mu^S>0$ and $\mu^B\geq 0$ are shear and bulk viscosity coefficients and $\mathbb I_d$ stands for the $d\times d$ identity matrix. Furthermore, the velocity gradient $\nabla u$ and the divergence of a $d\times d$--matrix valued function $\dvr \mathbb A$ are defined as 
\begin{equation}
\nabla u=(\partial_{x_j}u_i)_{i,j=1}^d,\ (\dvr\mathbb A)_i=\sum_{j=1}^d\partial_{x_j}\mathbb A_{i,j},\ i=1,\ldots, d.
\end{equation}
Before we give the precise definition of a weak solution to \eqref{NSEq} we collect hypotheses on the pressure. The relation between the pressure $p$ and the density $\rho$ of the fluid is so--called hard--sphere equation of state in the interval $[0,\bar\rho)$ 
\begin{equation}\label{PressAssump}
p\in C^1([0,\bar\rho)),\ p(0)=0,\ p'>0\text{ on }(0,\bar\rho),\ \lim_{s\to\bar\rho_-}p(s)=+\infty.
\end{equation}
We also define the pressure potential $P\in C^1([0,\bar\rho))$ as 
\begin{equation}\label{PresPot}
P(s)=s\int_{\frac{\bar\rho}{2}}^s\frac{p(z)}{z^2}
\end{equation}
and note that
\begin{equation}\label{PressPotId}
P'(s)s-P(s)=p(s),\ P''(s)=\frac{p'(s)}{s}\text{ for }s\in[0,\bar\rho).
\end{equation}

We are interested in  the well-accepted Carnahan-Starling equation of the state characterized by the properties in \eqref{PressAssump} and \eqref{PressConstr}. As explained in e.g. \cite{Song}, it is a suitable approximate equation of state for the fluid phase of the hard--sphere model. The derivation of such model was performed from a quadratic relation between the integer portions of the virial coefficients and their orders. This model is used for the study of the behavior of dense gases and liquids.  The interested reader can find more details regarding the model or its corrections (Percus-Yevick equation, Kolafa correction, Liu correction) in \cite{Carnahan-Starling, Hongqin,KLM04AEHS,KVB62CPES}.
Singular pressure laws of similar type appeared in modeling of various phenomena, including the collective motion, see \cite{MR3020033, MR2835410, maury2012}, and the traffic flow, see \cite{MR2438216, MR2366138}. Assymptotic limits for problems involving a singular pressure law were studied in  \cite{bresch2014, bresch2017, peza2015,
MR3974475}.
	  
	The study of existence of weak solutions to the compressible Navier--Stokes equations in the isentropic setting on a bounded domain goes back to the 
seminal work by Lions \cite{MR1637634} and the later improvement by Feireisl et al. \cite{MR1867887}. Concerning systems with a singular pressure law in a bounded domain with no--slip boundary conditions, the existence of weak solutions was shown by Feireisl et al. \cite{FeLuMa16} and Feireisl and Zhang \cite[Section 3]{FeZh10}. Recently, the existence of weak solutions to compressible Navier--Stokes equations with the hard--sphere pressure was investigated by Choe et al. \cite{MR3912678} for the case with a general inflow/outflow and in an exterior domain by Ne\v{c}asov\'a et al. \cite{NeNoRo}. Weak--strong uniqueness for the compressible Navier--Stokes equations with the hard pressure in periodic spatial domains was shown by Feireisl et al. \cite{FeLuNo18}.

This paper is motivated by the result in \cite{FeNeSu14} concerning the assymptotic limits for the compressible Navier--Stokes system in the isentropic setting on an expanding domain with ill--prepared initial data. Our second motivation comes from \cite{FeLuNo18}, where the modification of the relative entropy inequality, originaly derived for the isentropic regime by Feireisl et al. \cite{FeJiNo12}, was derived in the case of periodic boundary conditions. The aim of this paper is twofold. First, we want to derive of the relative entropy inequality for the Navier--Stokes problem in the setting with the hard--sphere pressure and no--slip boundary conditions.  Second, we want to study the asymptotic limit of the compressible Navier-Stokes system with the pressure obeying a hard--sphere equation of state on a domain expanding to the whole physical space $\eR^3$. 

The outline of the paper is as follows. Section \ref{intro} deals with the description of the problem, the meaning of weak solution to the problem, the statement of the main result of the paper and the derivation of the relative entropy inequality. 
Section \ref{sing} is devoted to the study of the assymptotic limit in the primitive compressible Navier--Stokes problem \eqref{NSEq} yielding the Euler incompressible equations in the whole physical space $\eR^3$ as the target system. Finally, in the Appendix we deal with renormalized solutions of the continuity equation adopted for a function $b$ satisfying \eqref{b} and \eqref{BNondec}.

\section{Definition of weak solutions and preliminaries}\label{intro}

We introduce the  definition of weak solutions as was done in \cite{FeLuNo18}.

\begin{Definition}\label{WSDef}
Let the following hypotheses be imposed on the initial data
\begin{equation}\label{InitDAssum}
\rho_0\in[0,\bar\rho)\text{ a.e. in }\Omega,\ \int_\Omega P(\rho_0)<\infty,\ \int_\Omega \rho_0|u_0|^2<\infty.
\end{equation}
A pair $(\rho,u)$ is said to be a finite--energy weak solution to \eqref{NSEq} if
\begin{itemize}
	\item $\rho\in [0,\bar{\rho})$ a.e. in $Q_T$, $\rho\in C_w([0,T];L^\gamma(\Omega))$ for any $\gamma>1$, $p(\rho)\in L^1(Q_T)$,
	\item  $u\in L^2(0,T;W^{1,2}_0(\Omega)^d)$, $\rho u\in C_w([0,T];L^2(\Omega)^d)$, $\rho|u|^2\in L^\infty(0,T;L^1(\Omega))$. 
	\item 
	The continuity equation 
	\begin{equation}
	\int_0^\tau\int_\Omega \rho\tder\phi+\rho u\cdot\nabla \phi=\int_\Omega \rho(\tau,\cdot)\phi(\tau,\cdot)-\int_\Omega\rho_0\phi(0,\cdot)
	\end{equation}
	is satisfied for any $\tau\in (0,T)$ and any test function $\phi\in C^1([0,T];C^1(\overline{\Omega}))$.
	\item The momentum equation 
	\begin{equation}\label{MomEqWeak}
	\begin{split}
	&\int_0^\tau\int_\Omega \rho u\tder\varphi+\left(\rho u\otimes u-\mathbb S(\nabla u)\right)\cdot\nabla \varphi+p(\rho)\dvr\varphi\\
	&=-\int_0^\tau\int_\Omega \rho f\cdot\varphi +\int_\Omega \rho(\tau,\cdot)\phi(\tau,\cdot)-\int_\Omega\rho_0\phi(0,\cdot)
	\end{split}
	\end{equation}
	is satisfied for any $\tau\in (0,T)$ and any test function $\varphi\in C^1_c([0,T]\times\Omega)^d$.
	\item 
	The continuity equation holds in the sense of renormalized solutions
	\begin{equation}
	\int_0^T\int_\Omega b(\rho)\tder\psi+b(\rho) u\cdot\nabla \psi+(b'(\rho)\rho-b(\rho))\dvr u\psi=0
	\end{equation}
	for any test function $\psi\in C^\infty_c(Q_T)$ and any function $b\in C^1([0,\bar\rho))$ satisfying
	\begin{equation}\label{b}
	|b'(s)|^2+|b(s)|^2\leq c(1+p(s))\text{ for some }c>0\text{ and any }s\in[0,\bar{\rho}).
	\end{equation}
	\item The energy inequality holds for a.a. $\tau\in(0,T)$:
	\begin{equation}
	\begin{split}
	\int_\Omega &\left(\frac{1}{2}\rho|u|^2+P(\rho)\right)(\tau)+\int_0^\tau\int_\Omega \mathbb S(\nabla u)\cdot\nabla u\\
	&\leq \int_\Omega \left(\frac{1}{2}\rho_0|u_0|^2+P(\rho_0)\right)+\int_0^\tau\int_\Omega \rho f\cdot u.
	\end{split}
	\end{equation}
\end{itemize}
\end{Definition}
\begin{Remark}
	If the class of admissible test functions in \eqref{MomEqWeak} is reduced to $C^\infty_c(Q_T)^3$ one can conclude that $\nabla p(\rho)\in X^*$, where $X=L^\frac{5}{2}(0,T;W_0^{1,\frac{5}{2}}(\Omega))\cap W^{1,2}(0,T;L^\frac{6}{5}(\Omega))$. Indeed, the regularity of $\rho$, $u$ and $\rho u$ specified in Definition \ref{WSDef} and 
	\begin{equation*}
	\rho u\otimes u\in L^\infty(0,T;L^1(\Omega))\cap L^1(0,T;L^3(\Omega))\subset L^\frac{5}{3}(Q_T),
	\end{equation*}
	where also the embedding $W^{1,2}(\Omega)$ into $L^6(\Omega)$\footnote{In fact for $\Omega\subset\eR^2$, $W^{1,2}(\Omega)$ is embedded into $L^q(\Omega)$ for any $q\in[1,\infty)$ but the better integrability will not bring any benefits in further analysis. For the sake of clarity we will not distinguish between the $2d$ and $3d$ case.} was applied, imply the regularity of the distribution $\nabla p$ provided that \eqref{MomEqWeak} is used for the expression of the duality $\int_0^T\langle \nabla p, \varphi\rangle$. Using the regularity of $\rho$, $u$, $\rho u$, $\rho u\otimes u$ and $\nabla p(\rho)$ and a density argument one can alternatively formulate the momentum equation \begin{equation}\label{MomEqWeakNablaP}
	\begin{split}
	&\int_0^\tau\int_\Omega \rho u\tder\varphi+\left(\rho u\otimes u-\mathbb S(\nabla u)\right)\cdot\nabla \varphi-\int_0^\tau \langle\nabla p(\rho),\varphi\rangle\\
	&=-\int_0^\tau\int_\Omega \rho f\cdot\varphi +\int_\Omega \rho(\tau,\cdot)\varphi(\tau,\cdot)-\int_\Omega\rho_0\varphi(0,\cdot)
	\end{split}
	\end{equation}
	for any $\tau\in (0,T)$ and any test function $\varphi\in X$.
\end{Remark}

For the purposes of this paper we define the relative entropy functional as
\begin{equation}
\mathcal E(\rho, u|r,U)(t)=\int_\Omega\left(\frac{1}{2}\rho|u-U|^2+P(\rho)-P(r)-P'(r)(\rho-r)\right)(t,x)\dx.
\end{equation}
The ensuing theorem deals with the global in time existence of a finite--energy weak solution to \eqref{NSEq} that satisfies a version of so called relative entropy inequality.
\begin{Theorem}\label{Thm:Main}
	Suppose $T>0$ and $\Omega\subset\eR^d$ $d=2,3$ be a bounded domain with $C^{2,\nu}$--boundary for some $\nu>0$. Let the pressure functional satisfies besides \eqref{PressAssump} also the constraint 
	\begin{equation}\label{PressConstr}
	\lim_{\rho\to\bar\rho_-}p(\rho)(\bar\rho-\rho)^\beta>0\text{ for some }\beta>\frac{5}{2}
	\end{equation}
	and the initial data satisfy \eqref{InitDAssum}.
	Let $b\in C^1([0,\bar\rho))$ be a nonnegative function such that
	\begin{equation}\label{BNondec}
	\begin{split}
	b,b'\text{ are nondecreasing on }[\bar\rho-\alpha_0,\bar\rho)\text{ for some }\alpha_0\in(0,\bar\rho),\\
	|b'|^\frac{5}{2}+|b|^\frac{5}{2}\leq c(1+p)\text{ on }[0,\bar\rho)\text{ for some }c>0.
	\end{split}
	\end{equation}
	Then there exists a finite--energy weak solution $(\rho,u)$ to \eqref{NSEq} in the sense of Definition \ref{WSDef}. Moreover, if $(r,U)\in C^1(\overline {Q_T})\times C^1([0,T];C^2(\overline\Omega)^d)$ with $U=0$ on $(0,T)\times\partial\Omega$ satisfies
	\begin{equation*}
		0<\inf_{\overline{Q_T}}r\leq \sup_{\overline{Q_T}}r<\bar\rho
	\end{equation*}
	then the relative entropy inequality 
	\begin{equation*}
	\begin{split}
		&\mathcal E(\rho,u|r,U)+\int_0^\tau\int_\Omega\left(\mathbb S(\nabla u)-\mathbb S(\nabla U)\right)\cdot\nabla (u-U)+\int_0^\tau\int_\Omega p(\rho)b(\rho)\\
		&\leq \mathcal E(\rho_0,u_0|r(0,\cdot),U(0,\cdot))+\int_0^\tau\mathcal R_1(t)+\mathcal R_2(t)\dt+\mathcal{R}_3(\tau)
	\end{split}
	\end{equation*}
	holds for a.a. $\tau\in(0,T)$. The remainder terms $\mathcal R_i$ $i=1,2,3$ read
	\begin{equation*}
		\begin{split}
		\mathcal R_1(t)=&	\int_\Omega \rho\left(\tder U+(u\cdot\nabla) U\right)\cdot(U-u)+\int_\Omega \mathbb S(\nabla U)\cdot \nabla(U-u)+\int_\Omega \rho f\cdot(u-U)\\
		&+\int_\Omega \left((r-\rho)\tder P'(r)+(rU-\rho u)\cdot \nabla P'(r)\right)+\int_\Omega \dvr U\left(p(r)-p(\rho)\right),\\
		\mathcal R_2(t)=& \int_\Omega p(\rho)\langle b(\rho)\rangle-\int_\Omega \rho u\otimes u\cdot\nabla \mathcal B\left(b(\rho)-\langle b(\rho)\rangle\right)+\int_\Omega \mathbb{S}(\nabla u)\cdot\nabla \mathcal B\left(b(\rho)-\langle b(\rho)\rangle\right)\\
		&-\int_\Omega\rho f\cdot \mathcal B\left(b(\rho)-\langle b(\rho)\rangle\right)+\int_\Omega \rho u\cdot\mathcal B\left(\dvr(b(\rho)u)-\langle \dvr(b(\rho)u \rangle \right)\\
		&+\int_\Omega\rho u\cdot \mathcal B\left((b'(\rho)\rho-b(\rho))\dvr u-\langle (b'(\rho)\rho-b(\rho))\dvr u\rangle\right)\\
		\mathcal R_3(\tau)=&\int_\Omega \rho u\cdot\mathcal B\left(b(\rho)-\langle b(\rho)\rangle\right)(\tau,\cdot)-\int_\Omega \rho_0u_0\cdot \mathcal B\left(b(\rho_0)-\langle b(\rho_0)\rangle\right).
		\end{split}
	\end{equation*}
\end{Theorem}
	We note that $\mathcal B$ stands for the Bogovskii operator and the notation $\langle g\rangle$=$\langle g,1\rangle$ is used whenever $g$ belongs to a dual space to a Banach space containig the element $1$. In particular, if $g\in L^p(\Omega)$ then $\langle g\rangle=\tfrac{1}{|\Omega|}\int_\Omega g$.
\begin{proof}[Proof of Theorem \ref{Thm:Main}]
The existence of a global in time weak solution to \eqref{NSEq} in the sense of Definition \ref{WSDef} can be shown by employing the standard approximation scheme for the compressible Navier-Stokes system. One adopts the regularization of the pressure from \cite{FeZh10}. The existence proof in the latter reference relies on a constraint \eqref{PressConstr} for $\beta\geq 3$. Later it turned out that the exponent $\beta>\frac{5}{2}$ in \eqref{PressConstr} is sufficient for the existence proof, cf. \cite{FeLuMa16}.

Therefore we now concentrate on proving the relative entropy inequality.
Following the arguments employed in \cite{FeJiNo12} we obtain that
\begin{equation}\label{1STPEI}
\mathcal E(\rho, u|r,U)(\tau)+\int_0^\tau\int_\Omega \mathbb S\left(\nabla (u-U)\right)\cdot\nabla (u-U)\leq \mathcal E(\rho_0, u_0|r(0,\cdot),U(0,\cdot))+ \int_0^\tau\mathcal{R}_1.
\end{equation}
Moreover, it was derived in \cite{FeLuNo18} that 
\begin{equation}
|b(s)|+|b'(s)|\leq c(1+p(s))^\frac{2}{5}\leq c(1+P(s))^\frac{2}{3},
\end{equation}
provided that \eqref{PresPot} and \eqref{PressConstr} are taken into account. The latter inequlities in combination with \eqref{InitDAssum} and \eqref{BNondec}$_2$ yield
\begin{equation}\label{BRhoInt}
\begin{split}
b(\rho_0), b'(\rho_0)&\in L^\frac{3}{2}(\Omega)\\
b(\rho), b'(\rho)&\in L^\infty(0,T;L^\frac{3}{2}(\Omega))\cap L^\frac{5}{2}(Q_T).
\end{split}
\end{equation}
Having \eqref{1STPEI} at hand we are left with the proof of 
\begin{equation}\label{PressRenId}
\int_0^\tau\int_\Omega p(\rho)b(\rho)=\int_0^\tau\mathcal R_2(t)\dt+\mathcal R_3(\tau).
\end{equation}
The idea of proving this identity is to employ $\mathcal B(b(\rho)-\langle b(\rho)\rangle)$ as a test function in the momentum equation. Unfortunately, the low regularity of the time derivative of the latter function excludes this possibility. Indeed, expressing the time derivative of $\mathcal B(b(\rho)-\langle b(\rho)\rangle)$ in terms of the renormalized continuity equation with the function $b$ we obtain
\begin{equation*}
\begin{split}
\tder \mathcal B(b(\rho)-\langle b(\rho)\rangle)=&\mathcal B\left(\tder \left(b(\rho)-\langle b(\rho)\rangle\right)\right)=\mathcal B\left(\dvr(b(\rho)u)-\langle\dvr(b(\rho)u)\rangle\right)\\
&+\mathcal B\left((b'(\rho)\rho-b(\rho))\dvr u-\langle (b'(\rho)\rho-b(\rho))\dvr u\rangle\right).
\end{split}
\end{equation*}
Taking into consideration $b'(\rho)\in L^\frac{5}{2}(Q_T)$, $\dvr u\in L^2(Q_T)$ and the continuity of $\mathcal B$ from $L^{\frac{10}{9}}(\Omega)\to W^{1,\frac{10}{9}}_0(\Omega)^d$ it follows that the last term in the latter identity belongs to $L^\frac{10}{9}(0,T;W^{1,\frac{10}{9}}_0(\Omega)^d)$, which does not imply the regularity required for the time derivative of a test function in \eqref{MomEqWeak}. In order to circumvent this obstacle, we consider a suitable regularization of the function $b$. Namely,
we define for $\alpha\in(0,\alpha_0)$ with $\alpha_0$ from \eqref{BNondec} the regularization $b_\alpha$ of $b$ as 
\begin{equation}\label{BRegul}
b_\alpha(s)=\begin{cases}
b(s)&s\in[0,\bar\rho-\alpha]\\
b(\bar\rho-\alpha)&s\in(\bar\rho-\alpha,\bar\rho).
\end{cases}
\end{equation} 
Next, considering the function $\varphi=\mathcal B(b_\alpha(\rho)-\langle b_\alpha(\rho)\rangle)$ we immediately deduce that $\varphi\in L^\infty(0,T;W^{1,q}_0(\Omega)^d)$ for any $q\in[1,\infty)$, cf. Lemma \ref{Lem:Bog}. Moreover, using the linearity of $\mathcal B$ we get 
\begin{equation*}
	\tder \varphi=\mathcal B(\tder b_\alpha(\rho)-\langle \tder b_\alpha(\rho)\rangle).
\end{equation*}
Employing \eqref{BAlphaRCEq} we have
\begin{equation}\label{TderPhi}
\tder \varphi=\mathcal B\bigl(\dvr(b_\alpha(\rho)u)+(b'_\alpha(\rho)\rho-b_\alpha(\rho))\dvr u-\langle (b'_\alpha(\rho)\rho-b_\alpha(\rho))\dvr u\rangle\bigr).
\end{equation}
We notice that $\langle\dvr(b_\alpha(\rho)u)\rangle=0$ provided $\dvr(b_\alpha(\rho)u)$ is understood as an element of $L^2(0,T;(W^{1,\frac{6}{5}}(\Omega))')$. To show this fact we consider a sequence $\{S_\varepsilon(b_\alpha(\rho))u\}$, where $S_\varepsilon$ is a mollifier with respect to the space variables. Applying properties of mollifiers, the facts that $b_\alpha(\rho)\in L^\infty(Q_T)$, $u\in L^2(0,T;L^6(\Omega))$ and the Lebesgue dominated convergence theorem it follows that 
\begin{equation}\label{MollBACnv}
S_\varepsilon(b_\alpha (\rho))u\to b_\alpha (\rho)u\text{ in }L^2(0,T;L^6(\Omega)).
\end{equation}
 Moreover, as $S_\varepsilon (b_\alpha(\rho))u$ posesses the vanishing trace on $\partial\Omega$, which is not clear for $b_\alpha(\rho)u$, there is a representation for the duality 
\begin{equation*}
\int_0^T\langle\dvr(S
_\varepsilon(b_\alpha(\rho))u),\phi\rangle_{(W^{1,\frac{6}{5}}(\Omega))'\times W^{1,\frac{6}{5}}(\Omega)}=-\int_0^T\int_\Omega S
_\varepsilon(b_\alpha(\rho))u\cdot\nabla\phi\text{ for any }\phi\in L^2(0,T;W^{1,
	\frac{6}{5}}(\Omega)).
\end{equation*}
Accordingly, by \eqref{MollBACnv} we conclude that $\{\dvr(S_\varepsilon(b_\alpha(\rho)u))\}$ is a Cauchy sequence in $L^2(0,T;(W^{1,\frac{6}{5}}(\Omega))')$ implying $\dvr(b_\alpha(\rho)u)\in L^2(0,T;(W^{1,\frac{6}{5}}(\Omega))')$ and the representation
\begin{equation}\label{FunctRepForm}
\int_0^T\langle\dvr(b_\alpha(\rho)u),\phi\rangle_{(W^{1,\frac{6}{5}}(\Omega))'\times W^{1,\frac{6}{5}}(\Omega)}=-\int_0^T\int_\Omega b_\alpha(\rho)u\cdot\nabla\phi\text{ for any }\phi\in L^2(0,T;W^{1,\frac{6}{5}}(\Omega)).
\end{equation}
The suitable choice of an arbitrary $\phi$ independent of the space variable yields $\langle\dvr(b_\alpha(\rho)u)\rangle=0$ a.e. in $(0,T)$.

Taking into account $b_\alpha(\rho),b'_\alpha(\rho),\rho\in L^\infty(Q_T)$ and $u\in L^2(0,T;L^6(\Omega))$ it follows that
\begin{equation*} 
\dvr(b_\alpha(\rho)u)+(b'_\alpha(\rho)\rho-b_\alpha(\rho))\dvr u-\langle (b'_\alpha(\rho)\rho-b_\alpha(\rho))\dvr u\rangle\in L^2(0,T; (W^{1,\frac{6}{5}}(\Omega))').
\end{equation*}
 Hence using Lemma \ref{Lem:Bog} it follows that $\tder\varphi\in L^2(0,T;L^6(\Omega)^d)$. Consequently, we obtain that $\varphi$ is an admissible test function in \eqref{MomEqWeakNablaP}. Employing $\varphi$ as a test function in \eqref{MomEqWeakNablaP} we infer
\begin{equation}\label{AlphaIdent}
	\int_0^\tau\int_\Omega p(\rho)b_\alpha(\rho)-I^\alpha_1=-\int_0^\tau\langle\nabla p(\rho), \mathcal{B}(b_\alpha (\rho)-\langle b_\alpha(\rho)\rangle\rangle =\sum_{j=2}^6 I^{\alpha}_{j}+J^{\alpha}(\tau),
\end{equation}
where
\begin{equation}
	\begin{split}
	I_1^\alpha=&\int_0^\tau\int_\Omega p(\rho)\langle b_\alpha(\rho)\rangle\\
	I^{\alpha}_{2}=&-\int_0^\tau\int_\Omega \rho u\otimes u\cdot\nabla \mathcal B(b_\alpha(\rho)-\langle b_\alpha(\rho)\rangle)\\
	I^{\alpha}_{3}=&\int_0^\tau\int_\Omega\mathbb S(\nabla u)\cdot \nabla\mathcal B(b_\alpha(\rho)-\langle b_\alpha(\rho)\rangle)\\
	I^{\alpha}_4=&-\int_0^\tau\int_\Omega \rho f\cdot\mathcal B(b_\alpha(\rho)-\langle b_\alpha(\rho)\rangle)\\
	I^{\alpha}_5=&\int_0^\tau\int_\Omega \rho u\cdot\mathcal B(\dvr(b_\alpha(\rho)u))\\
	I^{\alpha}_6=&\int_0^\tau\int_\Omega \rho u\cdot\mathcal B((b'_\alpha(\rho)\rho-b_\alpha(\rho))\dvr u-\langle (b'_\alpha(\rho)\rho-b_\alpha(\rho))\dvr u\rangle)\\
	J^{\alpha}(\tau)=&\int_\Omega \rho u\cdot\mathcal B(b_\alpha(\rho)-\langle b_\alpha(\rho)\rangle)(\tau,\cdot)-\int_\Omega \rho_0 u_0\cdot\mathcal B(b_\alpha(\rho_0)-\langle b_\alpha(\rho_0)\rangle).
	\end{split}
\end{equation}
The next task is the limit passage $\alpha\to 0_+$ in \eqref{AlphaIdent}. 
To this end, we use the following convergences as $\alpha\to 0_+$
\begin{equation}\label{BAlphaCnv}
\begin{alignedat}{2}
	b_\alpha (\rho)&\to b(\rho)&&\text{ in }L^\frac{5}{2}(Q_T)\text{ and a.e. in }Q_T,\\
	b'_\alpha(\rho)&\to b'(\rho)&&\text{ in }L^\frac{5}{2}(Q_T),\\
	b_\alpha(\rho_0)&\to b(\rho)&&\text{ in }L^\frac{3}{2}(\Omega),\\
		b'_\alpha(\rho_0)&\to b'(\rho)&&\text{ in }L^\frac{3}{2}(\Omega).
	\end{alignedat}
\end{equation} 
We note that the latter convergences follow by the definition of $b_\alpha$ in \eqref{BRegul}, the assumption that $b\in C^1([0,\bar\rho))$, \eqref{BRhoInt} and the Lebesgue dominated convergence theorem. Next, taking into account that $u\in L^2(0,T;L^6(\Omega)^d)$ we get as $\alpha\to 0_+$
\begin{equation}\label{BAlphaBPrAlphaCnv}
\begin{alignedat}{2}
b_\alpha(\rho)u&\to b(\rho)u&&\text{ in }L^\frac{10}{9}(0,T;L^\frac{30}{17}(\Omega)^d),\\
\dvr(b_\alpha(\rho)u)&\to\dvr(b(\rho)u)&&\text{ in }L^{10}(0,T; (W^{1,\frac{30}{13}}(\Omega))'),\\
\left(b'_\alpha(\rho)\rho-b_\alpha(\rho)\right)\dvr u&\to \left(b'(\rho)\rho-b(\rho)\right)\dvr u&&\text{ in }L^\frac{10}{9}(Q_T).
\end{alignedat}
\end{equation} 
Let us point out that \eqref{BAlphaBPrAlphaCnv}$_2$ follows from \eqref{BAlphaBPrAlphaCnv}$_1$ by repeating the procedure, which leads to \eqref{FunctRepForm}.
Moreover, by Lemma \ref{Lem:Bog} we conclude
\begin{equation}\label{BogCnv}
\begin{alignedat}{2}
\mathcal B([b_\alpha(\rho)]_0)&\to \mathcal B([b(\rho)]_0)&&\text{ in }L^\frac{5}{2}(0,T;W^{1,\frac{5}{2}}_0(\Omega)^d),\\
\mathcal B(\dvr (b_\alpha(\rho)u))&\to \mathcal B(\dvr(b(\rho)u))&&\text{ in }L^\frac{10}{9}(0,T;L^\frac{30}{17}(\Omega)^d),\\
\mathcal B\left([\left(b'_\alpha(\rho)\rho-b_\alpha(\rho)\right)\dvr u]_0\right)&\to \mathcal B\left([\left(b'(\rho)\rho-b(\rho)\right)\dvr u]_0\right)&&\text{ in }L^\frac{10}{9}(0,T;W^{1,\frac{10}{9}}(\Omega)^d),
\end{alignedat}
\end{equation}
where the notation $[g]_0=g-\langle g\rangle$ was used.
By \eqref{BAlphaCnv}$_1$ and the properties of $b$ that allow to consider $|\Omega|^{-1}p(\rho)\|b(\rho)\|_{L^\infty(0,T;L^1(\Omega))}$ as an integrable majorant for $p(\rho)\langle b_\alpha(\rho)\rangle$ we employ the Lebesgue dominated convergence theorem for the limit passage $\alpha\to 0_+$ in $I^\alpha_1$.
We use the convergence \eqref{BogCnv}$_1$ for the passage to the limit $\alpha\to 0_+$ in $I^\alpha_2$, $I^\alpha_3$, $I^\alpha_4$. Moreover, this convergence implies $\mathcal B(b_\alpha(\rho)-\langle b_\alpha(\rho)\rangle)(\tau)$ in $W^{1,\frac{5}{2}}_0(\Omega)^d$ for a.a. $\tau\in(0,T)$ allowing for the passage $\alpha\to 0_+$ in $J^\alpha(\tau)$. One also applies \eqref{BAlphaCnv}$_{3,4}$ in this limit passage. As $\rho u\in L^{10}(0,T;L^\frac{30}{13}(\Omega)^d)\subset L^\infty(0,T;L^2(\Omega)^d)\cap L^2(0,T;L^6(\Omega)^d)$ it follows that \eqref{BogCnv}$_2$ allows us to pass to the limit $\alpha\to 0_+$ in $I^\alpha_5$. Applying \eqref{BogCnv}$_3$, the Sobolev embedding, the fact that $\rho\in L^\infty(Q_T)$, $u\in L^\infty(0,T;L^2(\Omega)^d)$ we pass to the limit $\alpha\to 0_+$ in $I^\alpha_6$. Therefore \eqref{AlphaIdent} yields
\begin{equation*}
\begin{split}
	\int_0^\tau\int_\Omega p(\rho)b(\rho)=&\lim_{\alpha\to 0_+}\int_0^\tau\int_\Omega p(\rho)b_\alpha(\rho)=\lim_{\alpha\to 0_+}\left(\sum_{j=1}^6 I^\alpha_j+J^\alpha(\tau)\right)\\
	=&	\int_0^\tau\int_\Omega p(\rho)\langle b(\rho)\rangle-\int_0^\tau\int_\Omega \rho u\otimes u\cdot\nabla \mathcal B(b(\rho)-\langle b(\rho)\rangle)
	\\&+\int_0^\tau\int_\Omega\mathbb S(\nabla u)\cdot \nabla\mathcal B(b(\rho)-\langle b(\rho)\rangle)
	-\int_0^\tau\int_\Omega \rho f\cdot\mathcal B(b(\rho)-\langle b(\rho)\rangle)\\
	&+\int_0^\tau\int_\Omega \rho u\cdot\mathcal B(\dvr(b(\rho)u))
	\\
	&+\int_0^\tau\int_\Omega \rho u\cdot\mathcal B((b'(\rho)\rho-b(\rho))\dvr u-\langle (b'(\rho)\rho-b(\rho))\dvr u\rangle)
	\\
	&+\int_\Omega \rho u\cdot\mathcal B(b(\rho)-\langle b(\rho)\rangle)(\tau,\cdot)-\int_\Omega \rho_0 u_0\cdot\mathcal B(b(\rho_0)-\langle b(\rho_0)\rangle).
	\end{split}
\end{equation*}
The first equality follows by the Lebesgue monotone convergence theorem provided we take account the definition of $b_\alpha$ implying $b_{\alpha_1}\leq b_{\alpha_2}$ for $\alpha_1\geq \alpha_2$ and the pointwise convergence $p(\rho)b_\alpha(\rho)\to p(\rho)b(\rho)$.
\end{proof}
By the application of the following lemma we obtain further estimates from the relative entropy inequality. It states properties of a quantity related to the pressure potential. It is a version of \cite[Lemma 4.1 and (4.15)]{FeLuNo18}.
\begin{Lemma}\label{Lem:PointPPotEst}
Let the function be defined via \eqref{PresPot} where the function $p$ satisfies \eqref{PressAssump} and \eqref{PressConstr}. Let $\rho\in [0,\overline\rho)$ and $r\in (0,\overline\rho)$ be such that  $0<\alpha_0\leq r\leq \overline\rho-\alpha_0<\overline\rho$ for some $\alpha_0\in(0,\overline\rho)$. Then there exist $\alpha_1\in(0,\alpha_0)$ and a constant $c>0$ such that
\begin{equation}\label{PPotEst}
P(\rho)-P(r)-P'(r)(\rho-r)\geq
\begin{cases}
c(\rho-r)^2,&\text{ if }\rho\in(\alpha_1,\bar\rho-\alpha_1),\\
\frac{p(r)}{2},&\text{ if }\rho\in[0,\alpha_1],\\
\frac{P(\rho)}{2}>1,&\text{ if }\rho\in[\bar\rho-\alpha_1,\bar\rho).
\end{cases}
\end{equation}
Additionally, we have
\begin{equation}\label{PEst}
p(\rho)-p(r)-p'(r)(\rho-r)\leq
\begin{cases}
c(\rho-r)^2,&\text{ if }\rho\in(\alpha_1,\bar\rho-\alpha_1),\\
1+p'(r)r-p(r),&\text{ if }\rho\in[0,\alpha_1],\\
2p(\rho),&\text{ if }\rho\in[\bar\rho-\alpha_1,\bar\rho).
\end{cases}
\end{equation}
\end{Lemma}

\begin{Lemma}\label{Lem:PressPotEst}
	Let the function $r:Q_T\to\eR$ satisfy $0<\alpha_0\leq r\leq \bar\rho-\alpha_0$ and the function $\rho:Q_T\to\eR$ take values in $[0,\bar\rho)$. Let the function $P\in C^1([0,\bar\rho))$ be defined in \eqref{PresPot}, where the function $p$ satisfies \eqref{PressAssump} and \ref{PressConstr}. Then there is a constant $C>0$ such that for a.a. $t\in(0,T)$ 
	\begin{equation}
	\|(\rho-r)(t)\|^2_{L^2(\Omega)}\leq C\int_\Omega \left(P(\rho)-P(r)-P'(r)(\rho-r)\right)(t,x)\dx.
	\end{equation} 
	\begin{proof}
		We observe that for a fixed $t\in(0,T)$ 
		\begin{equation}
			\int_{\Omega}|\rho-r|^2(t,x)\dx=\sum_{i=1}^3\int_{\Omega_i}|\rho-r|^2(t,x)\dx,
		\end{equation}
		where \begin{align*}
		\Omega_1&=\{x\in\Omega:\rho(t,x)\in[0,\alpha_1]\},\\
		\Omega_2&=\{x\in\Omega:\rho(t,x)\in(\alpha_1,\bar\rho-\alpha_1]\},\\
		\Omega_3&=\{x\in\Omega:\rho(t,x)\in(\bar\rho-\alpha_1,\bar\rho)\},\\
		\end{align*}
		with $\alpha_1$ coming from Lemma \ref{Lem:PointPPotEst}.
		We have $|\rho-r|\leq 2\bar\rho$ a.e. in $Q_T$. By Lemma \ref{Lem:PressPotEst} it follows that
		\begin{align*}
		\int_{\Omega}|\rho-r|^2(t,x)\dx&\leq \frac{2(\bar\rho)^2}{p(\alpha_0)}\int_{\Omega_1} \frac{p(r)}{2}(t,x)\dx+ c^{-1}\int_{\Omega_2}c|\rho-r|^2(t,x)\dx+4(\bar\rho)^2\int_{\Omega_3} \frac{P(\rho)}{2}(t,x)\dx\\
		&\leq C\int_\Omega \left(P(\rho)-P(r)-P'(r)(\rho-r)\right)(t,x)\dx,
		\end{align*}
		where also the assumption that $p$ is increasing was taken into account.
	\end{proof}
\end{Lemma}

\section{Singular limit}\label{sing}
We consider the scaled system with parameters $\nu>0$, $\varepsilon>0$ and $R>0$ satisfying
\begin{equation}\label{ScaledNSEq}
\begin{alignedat}{2}
\tder \rho+\dvr(\rho u)=&0&&\text{ in }(0,T)\times\Omega_R,\\
\tder(\rho u)+\dvr(\rho u\otimes u)-\nu\dvr \mathbb{S}(\nabla u)+\varepsilon^{-2}\nabla p(\rho)=&\rho f&&\text{ in }(0,T)\times\Omega_R,\\
\rho(0,\cdot)=\rho_0,\ u(0,\cdot)=&u_0&&\text{ in }\Omega_R,\\
u=&0&&\text{ on }(0,T)\times\partial\Omega_R,
\end{alignedat}
\end{equation}
where the behavior of a domain $\Omega_R\subset\eR^3$ will be speciefied later. We assume that
\begin{equation}\label{SDef}
\mathbb S(\nabla u)=\mu\left(\nabla u+(\nabla u)^\top-\frac{2}{3}\dvr u\mathbb I_3\right).
\end{equation}
The formal identification of the limit system when $\varepsilon,\nu\to 0$ and $R\to\infty$ in \eqref{ScaledNSEq} yields that a sequence of solutions $(\rho,u)=(\rho_{\varepsilon,\nu,R},u_{\varepsilon,\nu,R})$ to \eqref{ScaledNSEq} converges in a certain sense to $(\varrho,v)$, where $\varrho$ is a positive constant and $v$ is a strong solution to the incompressible Euler system:
\begin{equation}\label{IES}
\begin{split}
\tder v+v\cdot\nabla v +\nabla\Pi=&0,\ \dvr v=0,\\ v(0)=v_0&=H(u_0),
\end{split}
\end{equation}

\noindent whose properties are summarized in Lemma \ref{Lem:Eul}, 
and where $H$ denotes the standard Helmholtz projection. 

Assuming that the dependence of the pressure on the density is given by a hard--sphere equation of state and the initial data are ill-prepared the goal of this section is the rigorous proof of the above described formal process. 

Before the precise formulation of the main theorem we describe the geometry of the physical space. We consider a family of expanding domains $\{\Omega_R\}$ with the following properties
\begin{equation}\label{SmoothOm}
\Omega_R\subset\eR^d\text{ is simply connected, bounded }C^{2,\nu}\text{domain uniformly for }R\to\infty,
\end{equation}
\begin{equation}\label{StarShapeOm}
\Omega_R \text{ is star-shaped with respect to the ball }B(0,R)=\{x\in\eR^d:|x|<R\}
\end{equation}
there is a constant $D>0$ such that
\begin{equation}\label{BoundOmLimits}
\partial\Omega_R\subset\{x\in\eR^d:R<|x|<R+D\}.
\end{equation}
\begin{Theorem}
	Let the pressure function $p$ satisfy assumption \eqref{PressAssump} and additionaly $p\in C^2((0,\overline\rho))$ and the pressure potential be defined via \eqref{PresPot}. Let $\{\Omega_R\}_{R>1}$ be a family of uniformly $C^{2,\nu}$ domains for which \eqref{SmoothOm}, \eqref{StarShapeOm} and \eqref{BoundOmLimits} hold. Let the positive constants $D$, $\varrho$ be given and $\varepsilon_0>0$ be such that
	\begin{equation}\label{EpsZCh}
	D^{-1}<\varrho-\varepsilon_0D, \varrho+\varepsilon_0 D<\bar\rho.
	\end{equation}
	 Let $\varepsilon\in(0,\varepsilon_0)$ be fixed and $(\rho,u)$ be a finite energy weak solution of system \eqref{ScaledNSEq} emanating from the inital data 
	\begin{equation}\label{InitialData}
	\rho(0,\cdot)=\rho_{0,\varepsilon}=\varrho+\varepsilon\rho^{(1)}_{0,\varepsilon},\ u(0,\cdot)=u_{0,\varepsilon}
	\end{equation}
	with
	\begin{equation}\label{InitDBound}
	\|u_{0,\varepsilon}\|_{L^2(\eR^3)}+\|\rho^{(1)}_{0,\varepsilon}\|_{L^2(\eR^3)}+\|\rho^{(1)}_{0,\varepsilon}\|_{L^\infty(\eR^3)}\leq D.
	\end{equation}
	In addition, let 
	\begin{equation}\label{RadCond}
	R>D+\frac{\sqrt{p'(\varrho)}}{\varepsilon}T.
	\end{equation}
	Furthermore, assume that there are functions $u_0\in C^m(\eR^3;\eR^3),\rho^{(1)}_0\in C^m(\eR^d)$, $m\geq 4$ supported in $B(0,D)$ such that
	\begin{equation}\label{LimInitDBound}
	\|u_0\|_{C^m(\eR^3)}+\|\rho^{(1)}_0\|_{C^m(\eR^3)}\leq D.
	\end{equation}
	Let $v$ be a strong solution to the incompressible Euler system \eqref{IES} in $(0,T_{max})\times\eR^3$ and $T\in(0,T_{max})$. Let $(s,\Psi)$ be the solution of the acoustic system
	\begin{equation}\label{AcSys}
	\begin{split}
	\varepsilon\tder s+\varrho \Delta\Psi=&0,\\
	\varepsilon \tder\nabla \Psi+\frac{p'(\varrho)}{\varrho}\nabla s=&0
	\end{split}
	\end{equation}
	in $(0,T)\times\eR^d$ and suplemented with the initial data
	\begin{equation}\label{ACSysInitD}
	s(0,\cdot)=\rho^{(1)}_0,\ \nabla\Psi(0,\cdot)=\nabla \Psi_0=u_0-H(u_0).
	\end{equation}
	Then there is $\varepsilon_1>0$ such that 
	\begin{equation}\label{FinIneq}
	\begin{split}
	\int_{\Omega_R}&\rho|u-\nabla\Psi-v|^2(\tau,\cdot)+\left\|\frac{\rho(\tau,\cdot)-\varrho}{\varepsilon}-s(\tau,\cdot)\right\|^2_{L^2(\Omega_R)}\\
	&\leq \Bigl(c(D,T)(\varepsilon^\alpha+R^{-1}+\nu+\varepsilon^2(1+R^{-2})+\varepsilon\nu^{-1})+c_2\left(\|u_{0,\varepsilon}-u_0\|^2_{L^2(\Omega_R)}+\|\rho^{(1)}_{0,\varepsilon}-\rho^{(1)}_0\|^2_{L^2(\Omega_R)}\right)\Bigr)\\
	&\times\exp\left(c(D,T)\left(1+\varepsilon^2\nu+\varepsilon^\frac{4}{3}\nu^{-1}(1+R^{-4})+\varepsilon^2+R^{-2}+\varepsilon^2R^{-2}\right)\right.\\
	&\left.+\frac{ c\varepsilon^2}{\overline\rho-\varrho-\varepsilon_0 D}(\nu^{-1}R^{-1}+c\nu^\frac{1}{2}R^{-\frac{3}{2}}+1)\right).
	\end{split}
	\end{equation}
	for any $\varepsilon\in \left(0,\min\{\varepsilon_0,\varepsilon_1\}\right)$, any $\tau\in[0,T]$ and any $\alpha\in(0,1)$.
\end{Theorem}

\begin{Corollary}
	Assuming that $R,\nu$ are dependent on $\varepsilon$ and are such that  $R(\varepsilon)\to\infty,\ \varepsilon R(\varepsilon)\to\infty,\ \nu(\varepsilon)\to 0$,\ $\varepsilon\nu^{-1}(\varepsilon)\to 0$, $u_{0,\varepsilon}\to u_0$ in $L^2(\eR^3)$ and $\rho^{(1)}_{0,\varepsilon}\to\rho^{(1)}_0$ in $L^2(\eR^3)$ as $\varepsilon\to 0_+$ estimate \eqref{FinIneq} yields the uniform in time convergence of $u$ towards the solution $v$ corrected by the oscilatory component $\nabla\Psi$ and the convergence of the difference $\rho-\varrho$ scaled by the factor $\varepsilon^{-1}$ towards the oscilatory component $s$.
\end{Corollary}

We point out that the relative energy inequality from the previous section holds also for the scaled system in \eqref{ScaledNSEq} with the factor $\varepsilon^{-2}$ in front of the pressure and the pressure potential. Therefore for the relative entropy 
\begin{equation}
\mathcal E(\rho,u|r,U)(\tau)=\frac{1}{2}\int_{\Omega_R}(\rho|u-U|^2)(\tau,\cdot)+\varepsilon^{-2}(P(\rho)-P(r)-P'(r)(\rho-r))(t,\cdot)\dx
\end{equation}
we obtain 
\begin{equation}\label{REIEps}
\begin{split}
\mathcal E(\rho, u|r,U)(\tau)&+\nu\int_0^\tau\int_{\Omega_R}\left(\mathbb S(\nabla u)-\mathbb S(\nabla U)\right)\cdot\nabla(u-U)+\varepsilon^{-2}\int_0^\tau\int_{\Omega_R}p(\rho)b(\rho)\\
&\leq \mathcal E(\rho_0, u_0|r(0,\cdot),U(0,\cdot))+\int_0^\tau \mathcal R_1(t)+\mathcal R_2(t)\dt+\mathcal R_3(\tau),
\end{split}
\end{equation}
where 
\begin{align*}
\mathcal R_1(t)=&	\int_\Omega \rho\left(\tder U+(u\cdot\nabla) U\right)\cdot(U-u)+\nu\int_\Omega \mathbb S(\nabla U)\cdot\nabla (U-u)\\
&+\varepsilon^{-2}\int_\Omega (r-\rho)\tder P'(r)+\varepsilon^{-2}\int_{\Omega_R}(rU-\rho u)\cdot \nabla P'(r)+\varepsilon^{-2}\int_\Omega \dvr U\left(p(r)-p(\rho)\right),\\
\mathcal R_2(t)=& \varepsilon^{-2}\int_\Omega p(\rho)\langle b(\rho)\rangle-\int_\Omega \rho u\otimes u\cdot\nabla \mathcal B\left(b(\rho)-\langle b(\rho)\rangle\right)+\int_\Omega \mathbb{S}(\nabla u)\cdot\nabla \mathcal B\left(b(\rho)-\langle b(\rho)\rangle\right)\\
&+\int_\Omega \rho u\cdot\mathcal B\left(\dvr(b(\rho)u)\right)+\int_\Omega\rho u\cdot \mathcal B\left((b'(\rho)\rho-b(\rho))\dvr u-\langle (b'(\rho)\rho-b(\rho))\dvr u\rangle\right)\\
\mathcal R_3(\tau)=&\int_\Omega \rho u\cdot\mathcal B\left(b(\rho)-\langle b(\rho)\rangle\right)(\tau,\cdot)-\int_\Omega \rho_0u_0\cdot \mathcal B\left(b(\rho_0)-\langle b(\rho_0)\rangle\right).
\end{align*}

The rest of the section is devoted to the proof of Theorem \ref{Thm:Main}. It consists of three steps. First, taking into account the fact that we are in the situation with ill--prepared data, we choose properly test functions in the relative entropy inequality. Second, the relation between the values of the relative entropy functional at the time from the given interval and the initial value is deduced. At last, the Gronwall type argument is employed for the evaluation of the distance between the solutions of the primitive and target systems by means of the relative entropy functional. Moreover, the estimate of the rate of convergence expressed in terms of characteristic numbers and the radius of the expanding domains. We note that the convergence result is path dependent, i.e., there is a specific fashion in which the characteristic numbers and the radius of the expanding domain are interrelated.

The following two subsections deal with preparatory work that will justify our choice of test functions in the relative entropy inequality and also helps us in further estimates. The third subsection contains a collection of estimates that are independent of parameters $\varepsilon$, $\nu$ and $R$.
\subsection{Acoustic system}
This subsection is devoted to some properties of a solution to \eqref{AcSys} endowed with initial data \eqref{ACSysInitD}. Following the steps in \cite[Section 8.6]{FeNo17} it is possible to show that the solution of \eqref{AcSys} admits the finite speed of propagation $\frac{\sqrt{p'(\varrho)}}{\varepsilon}$. Therefore the solution of \eqref{AcSys} satisfies
\begin{equation}\label{ConstACW}
\nabla \Psi(t,x)=\nabla\Psi_0(x),\ \Delta\Psi(t,x)=s(t,x)=0\text{ for }t\geq 0,\ |x|>D+\frac{\sqrt{p'(\varrho)}}{\varepsilon}t.
\end{equation}
The physical assumption that the acoustic waves do not reach the boundary $\partial\Omega_R$ in the time lap $[0,T]$ is expressed by condition \eqref{RadCond}.

The conservation of energy of system \eqref{AcSys} is expressed in the form
\begin{equation}\label{EnConservRD}
\frac{\dd}{\dt}\left(p'(\varrho)\|s\|^2_{L^2(\eR^3)}+\varrho^2\|\nabla\Psi\|^2_{L^2(\eR^3)}\right)=0.
\end{equation}
Furthermore, the solution to \eqref{AcSys} obeys the higher energy estimates 
\begin{equation}\label{AcEnEst}
\|s(\tau)\|^2_{W^{k,2}(\eR^3)}+\|\nabla\Psi(\tau)\|^2_{W^{k,2}(\eR^3)}\leq c\left(\|\rho^{(1)}_0\|^2_{W^{k,2}(\eR^3)}+\|\nabla\Psi_0\|^2_{W^{k,2}(\eR^3)}\right)\text{ for any }\tau>0\text{ and }k=0,1,2,\ldots.
\end{equation}
and also the following estimates of Strichartz type
\begin{equation}\label{AcLpqEst}
\begin{split}
\|s(\tau)\|_{W^{k,q}(\eR^3)}&+\|\nabla\Psi(\tau)\|_{W^{k,q}(\eR^3)}\leq c(p,q)\left(1+\frac{\tau}{\varepsilon}\right)^{\frac{1}{q}-\frac{1}{p}}\left(\|\rho^{(1)}_0\|_{W^{k+4,p}(\eR^3)}+\|\nabla\Psi_0\|_{W^{k+4,p}(\eR^3)}\right)\\&\text{ for any }\tau>0,\ k=0,1,2,\ldots, \text{ and }\frac{1}{p}+\frac{1}{q}=1,\ p\in(1,2]
\end{split}
\end{equation}
that follow from results in \cite[Section 3]{Str70a}, cf. \cite[Section 1.1]{RuzSm10}, by a suitable rescaling in the time variable.

\subsection{Correctors}
For fixed $R>0$ we define the corrector $w_R$ as
\begin{equation}\label{WRDef}
w_R=-\chi_R(v+\nabla\Psi_0)
\end{equation}
where $\chi_R\in C^\infty(\eR^3)$, $0\leq \chi_R\leq 1$, $\chi_R(x)=1$ if $\dist(x,\partial\Omega_R)\leq\frac{1}{2}\min\{\dist(B_R,\partial\Omega),\dist(B_{R+D},\partial\Omega)\}$ and the support of $\chi_R$ is contained in $B_{R+D}\setminus B_R$. 
Before stating estimates involving the corrector $w_R$ we focus on properties of functions $v$ and $\nabla\Psi_0$.
Namely, we shortly discuss the fact that $\Psi_0$ and $\operatorname{curl} v$ are harmonic functions in the exterior of a ball with enough large radius.  Employing the Biot-Savart law, cf. \cite[Section 2.4.1]{MaBe02}, for the expression of $v$ we get
\begin{equation*}
v=-\operatorname{curl}\Delta^{-1}\operatorname{curl} v
\end{equation*}
where $\Delta^{-1}\operatorname{curl} v$ is obtained as the convolution of the Newtonian potential with $\operatorname{curl} v$. As the support of $u_0$ is assumed to be compact in $B(0,D)$, it follows that the support of $\operatorname {curl}v_0$, $\Delta\Psi_0$ respectively, is also compact in $B(0,D)$. Moreover, as $v$ is a smooth solution of the Euler system, the quantity $\operatorname{curl}v$ obeys a transport equation with a compactly supported initial datum. The latter implies that $\Delta^{-1}\operatorname{curl}v$ is a harmonic function in the exterior of the ball $B(0,R)$ for $R>D+T\|v\|_{L^\infty((0,T)\times\eR^3)}$.
Taking into acount the fact that both $v$ and $\nabla\Psi_0$ are derivatives of functions that are harmonic outside a ball $B(0,R)$ it follows that
\begin{equation*}
|\nabla^k \nabla\Psi_0(x)|,|\nabla^k v(x)|\leq c|x|^{-2-k}\text{ for }x\in\eR^3\setminus\overline{B(0,R)}\text{ and }k=0,1,2.
\end{equation*}
Hence we conclude
\begin{equation}\label{CorrEst}
\|\tder w_R(t)\|_{L^p(\eR^3)}+\|w_R(t)\|_{W^{2,p}(\eR^3)}\leq c R^{2(\frac{1}{p}-1)}
\end{equation}
as $|B_{R+D}\setminus B_R|\leq c DR^2$ provided that $R>1$.

\subsection{Uniform estimates}
We observe that setting $(r,U)=(\varrho,0)$ in relative entropy inequality \eqref{REIEps} we conclude
\begin{equation}\label{EpsInd}
\begin{split}
	\|\sqrt\rho u\|_{L^\infty(0,T;L^2(\Omega_R))}&\leq c,\\
	\left\|\frac{\rho-\varrho}{\varepsilon}\right\|_{L^\infty(0,T;L^2(\Omega_R))}&\leq c,\\
	\nu^\frac{1}{2}\|\nabla u\|_{L^2(0,T;L^2(\Omega_R))}&\leq c
\end{split}	
\end{equation}
where the constant $c$ is independent of $\varepsilon,\nu,R$.
We point out that the latter bound follows by the Korn inequality, see \cite[Theorem 11.22]{FeNo17}, provided that the extension of the function $u$  by zero in $(0,T)\times\left(\eR^3\setminus\Omega_R\right)$ is considered. Moreover, taking into account Lemma \ref{Lem:PointPPotEst} it follows from \eqref{REIEps} that there is $\alpha_1\in(0,\overline\rho)$ such that
\begin{equation}\label{PPotEpsEst}
\esssup_{t\in(0,T)}\left(\|\chi_{\{\rho(t,\cdot)\in(0,\alpha_1)\}}\|_{L^1(\Omega_R)}+\|P(\rho)(t,\cdot)\chi_{\{\rho(t,\cdot)>\overline\rho-\alpha_1\}}\|_{L^1(\Omega_R)}\right)\leq c\varepsilon^2.
\end{equation}

\subsection{Convergence}
Let us begin the proof of inequality \eqref{FinIneq} by specifying of the value of $\varepsilon_1$. Since our intention is to set $r=\varrho+\varepsilon s$ in the relative entropy inequality \eqref{REIEps}, we need 
\begin{equation}\label{RBound}
0\leq \varrho+\varepsilon s<\overline\rho
\end{equation} 
to have $P(r)$ well defined. To this end we get by the Sobolev embedding, \eqref{AcEnEst} and \eqref{LimInitDBound}
\begin{equation*}
\begin{split}
\|s\|_{L^\infty(0,T;L^\infty(\eR^3))}\leq c \|s\|_{L^\infty(0,T;W^{3,2}(\eR^3))}
\leq  c D(|\supp \rho^{(1)}_0|+|\supp u_0|).
\end{split}
\end{equation*}
Therefore taking 
\begin{equation}\label{Eps1Def}
\varepsilon_1=\min\{\overline\rho-\varrho,\varrho \}\left(c D(|\supp \rho^{(1)}_0|+|\supp u_0|)\right)^{-1},
\end{equation}
we conclude the validity of \eqref{RBound} for any $\varepsilon\in(0,\varepsilon_1)$. 
We set $(r,U)=\left(\varrho+\varepsilon s,v+\nabla\Psi+w_R\right)$ in the relative entropy inequality. Such a pair is admissible in \eqref{REIEps} as the boundary condition $U=0$ on $\partial\Omega$ is satisfied due to the definition of the corrector $w_R$ in \eqref{WRDef}.

Moreover, taking into account the fact that $P''(z)=\frac{p'(z)}{z}$ we infer for the initial data given in \eqref{InitialData} that
\begin{align*}
\mathcal E\left(\rho,u|r,U\right)(0)=&\int_{\Omega_R}\frac{1}{2}\rho_{0,\varepsilon}|u_{0,\varepsilon}-H(u_0)-\nabla\Psi_0-w_R(0,\cdot)|^2\\&+\varepsilon^{-2}\int_{\Omega_R}\left(P(\varrho+\varepsilon\rho^{(1)}_{0,\varepsilon})-P(\varrho+\varepsilon\rho^{(1)}_0)-\varepsilon P'(\varrho+\varepsilon\rho^{(1)}_{0})(\rho^{(1)}_{0,\varepsilon}-\rho^{(1)}_{0})\right)\\
\leq &c\|u_{0,\varepsilon}-u_0\|^2_{L^2(\Omega_R)}+\|w_R(0)\|^2_{L^2(\Omega_R)}+K\|\rho^{(1)}_{0,\varepsilon}-\rho^{(1)}_0\|^2_{L^2(\Omega_R)},
\end{align*}
where $K=\max_{z\in[\varrho,\varrho+\varepsilon_0D]}\frac{p'(z)}{z}$. Since $p\in C^1([0,\varrho+\varepsilon_0D])$, the quantity $K$ is finite and obviously independent of $\varepsilon\in(0,\varepsilon_0)$.
From now on we use the following notation for the integrals involved in terms $\mathcal R_1$, $\mathcal R_2$, $\mathcal R_3$.
\begin{equation*}
\int_0^t\mathcal R_1=\sum_{j=1}^5I_j,
\int_0^t\mathcal R_2=\sum_{j=6}^{10}I_j,
\mathcal R_3 =I_{11}+I_{12}
\end{equation*}
We proceed by estimating $I_j$'s in terms of the relative entropy functional and terms involving powers of quantities $\varepsilon$, $\nu$, $\frac{1}{R}$.
First, we rewrite
\begin{align*}
I_1=&-\int_0^t\int_{\Omega_R} \rho\left((u-U)\cdot\nabla\right) U\cdot(U-u)+\int_0^t\int_{\Omega_R} \rho(\tder U+(U\cdot\nabla) U)\cdot(U-u)\\
=&-\int_0^t\int_{\Omega_R}\rho\left((u-U)\cdot\nabla\right) U\cdot(U-u)\\&+\int_0^t\int_{\Omega_R}\rho(U-u)\cdot (\tder v+(v \cdot\nabla) v)+\int_0^t\int_{\Omega_R}\rho(U-u)\cdot\tder w_R+\int_0^t\int_{\Omega_R}\rho(U-u)\cdot\tder\nabla\Psi\\
&+\int_0^t\int_{\Omega_R}\rho((v+\nabla\Psi+w_R)\cdot\nabla)(\nabla\Psi+w_R)\cdot(U-u)\\
&+\int_0^t\int_{\Omega_R}\rho((\nabla\Psi+w_R)\cdot\nabla) v\cdot(U-u)=\sum_{k=1}^6 J_k.
\end{align*}
We immediately see that
\begin{equation}\label{J1Est}
\begin{split}
|J_1|\leq& \int_0^t\|v+\nabla\Psi+w_R\|_{L^\infty(\Omega_R)}\mathcal E(\rho,u|r,U)\leq c\int_0^t(\|v\|_{W^{3,2}(\eR^3)}+\|\Psi\|_{W^{4,2}(\eR^3)}+R^{-2})\mathcal E(\rho,u|r,U)\\
\leq& c(D,T)\int_0^t\mathcal E(\rho,u|r,U)
\end{split}
\end{equation}
by the Sobolev embedding and \eqref{AcEnEst}.
Using the Euler system for $v$, the weak formulation of the continuity equation for $\rho$, $U=0$ on $\partial\Omega_R$ and $\dvr v=0$ in $(0,T)\times\eR^3$ it follows that
\begin{align*}
J_2=&\int_0^t\int_{\Omega_R}\rho(u-U)\cdot\nabla \Pi= -\varepsilon \int_0^\tau\int_{\Omega_R}\frac{\rho-\varrho}{\varepsilon}\tder\Pi+\varepsilon\left[\int_{\Omega_R}\frac{\rho-\varrho}{\varepsilon}\Pi\right]_{t=0}^{t=\tau}-\varepsilon\int_0^\tau\int_{\Omega_R}\frac{\rho-\varrho}{\varepsilon}U\cdot\nabla\Pi\\
&+\varrho\int_0^\tau\int_{\Omega_R}(\dvr w_R+\Delta\Psi)\Pi=J_{2,a}+J_{2,b}+J_{2,c}+J_{2,d}.
\end{align*}
Employing \eqref{EpsInd}$_2$ it follows that
\begin{equation}\label{J2ABCEst}
\begin{split}
|J_{2,a}|+|J_{2,b}|+|J_{2,c}|&\leq c\varepsilon(\|\tder\Pi\|_{L^1(0,T;L^2(\eR^3))}+\|\Pi\|_{L^\infty(0,T;L^2(\eR^3))}+\|U\|_{L^\infty(0,T;L^\infty(\Omega_R))}\|\nabla \Pi\|_{L^1(0,T;L^2(\eR^3))})\\
&\leq c\varepsilon.
\end{split}
\end{equation}
where the last inequality is obtained with help of Lemma \ref{Lem:Eul}. Using this lemma in combination with \eqref{WRDef} and \eqref{AcLpqEst} yield
\begin{equation}\label{J2DEst}
\left|J_{2,d}\right|\leq \varrho\left(\|w_R\|_{L^\infty(0,T;W^{1,p}(\Omega_R))}+\|\Delta\Psi\|_{L^\infty(0,T;L^p(\Omega_R))}\right)\|\Pi\|_{L^1(0,T;L^{p'}(\Omega_R))}\leq c\left(R^{-1}+\varepsilon^{1-\frac{2}{p}}\right)
\end{equation}
for any $p>2$.
Obviously, by \eqref{EnConservRD}, \eqref{EpsInd}$_1$ and \eqref{CorrEst} we conclude 
\begin{equation}
|J_3|\leq \left( \overline\rho\|U\|_{L^\infty(0,T;L^2(\eR^3))}+\sqrt{\overline\rho}\|\sqrt\rho u\|_{L^\infty(0,T;L^2(\Omega_R))}\right)\|\tder w_R\|_{L^1(0,T;L^2(\Omega_R))}\leq cR^{-1}.
\end{equation}
Next, we rewrite using \eqref{RadCond} and \eqref{ConstACW}
\begin{equation}\label{J3Split}
\begin{split}
J_4=& \int_0^\tau\int_{\Omega_R}(\rho-\varrho)v\cdot\partial_t\nabla\Psi+\varrho\int_0^\tau\int_{\Omega_R}v\cdot\partial_t\nabla\Psi+\int_0^\tau\int_{\Omega_R}(\rho-\varrho)\nabla\Psi\cdot\tder\nabla\Psi+\frac{\varrho}{2}\left[\int_{\eR^3}|\nabla\Psi|^2\right]_{t=0}^{t=\tau}\\
&+\int_0^{\tau}\int_{\Omega_R}\rho w_R\cdot\tder\nabla\Psi-\int_0^\tau\int_{\Omega_R}\rho u\cdot\partial_t\nabla\Psi =J_{4,a}+J_{4,b}+J_{4,c}+J_{4,d}+J_{4,e}+J_{4,f}.
\end{split}
\end{equation}
Employing equation \eqref{AcSys}$_2$, the regularity of $v$ and estimate \eqref{AcLpqEst} we obtain
\begin{equation}\label{J3Est}
\begin{split}
|J_{4,a}|+|J_{4,c}|&\leq  \frac{p'(\varrho)}{\varrho}\varepsilon^{-1}\|\rho-\varrho\|_{L^\infty(0,T;L^2(\Omega_R))}\left(\|v\|_{L^\infty(0,T;L^{q_1}(\Omega_R))}+\|\nabla\Psi\|_{L^\infty(0,T;L^{q_1}(\Omega_R))}\right)\|\nabla s\|_{L^1(0,T;L^{{q_2}}(\Omega_R))}\\
&\leq c \varepsilon^{1-\frac{2}{q_2}}
\end{split}
\end{equation}
for any $q_1,q_2>2$ such that $\frac{1}{q_1}+\frac{1}{q_2}=\frac{1}{2}$. As \eqref{RadCond} and \eqref{ConstACW} imply $\Psi=\Psi_0$ in $(0,T)\times\left(\eR^3\setminus B(0,R)\right)$ and $\Psi_0$ is independent of time, it follows that $\eR^3$ can be taken as the domain of integration in $J_{4,d}$ and that 
\begin{equation*}
J_{4,b}=\int_0^\tau\int_{\partial\Omega_R}v\cdot n\tder\Psi_0=0,
\end{equation*}
where $\dvr v=0$ in $(0,T)\times\eR^3$ was also applied. Using \eqref{CorrEst} and \eqref{AcSys}$_2$ we conclude
\begin{equation}
|J_{4,e}|\leq \overline\rho\|w_R\|_{L2(0,T;L^2(\eR^3))}\varepsilon^{-1}\|\nabla s\|_{L^2(0,T;L^2(\eR^3))}\leq \frac{c}{\varepsilon R}.
\end{equation}
Next, using bound \eqref{EpsInd}$_{1}$, the fact that $\|v\|_{L^\infty(0,T;W^{1,\infty}(\eR^3))}$ is finite, see Lemma \ref{Lem:Eul} and \eqref{CorrEst} it follows that
\begin{equation}\label{J567Est}
\begin{split}
|J_5|+|J_6|\leq& \left(\|\rho u\|_{L^\infty(0,T;L^2(\Omega_R))}+\overline\rho\|U\|_{L^\infty(0,T;L^2(\Omega_R))}\right)\\ &\times\left(\|v\|_{L^\infty(0,T;W^{1,q_1}(\Omega_R))}+\|\nabla\Psi\|_{L^\infty(0,T;L^{q_1}(\Omega_R))}+\|w_R\|_{L^\infty(0,T;L^{q_1}(\Omega_R))}\right)\\
&\times\left(\|\Psi\|_{L^1(0,T;W^{2,q_2}(\Omega_R))}+\|w_R\|_{L^1(0,T;W^{1,q_2}(\Omega_R))}\right)\leq c\left(\varepsilon^{1-\frac{2}{q_2}}+R^{-2\left(1-\frac{1}{q_2}\right)}\right)
\end{split}
\end{equation}
for any $q_1,q_2> 2$ such that $\frac{1}{q_1}+\frac{1}{q_2}=\frac{1}{2}$. 

Before estimating the term $I_2$, we note that by the Korn inequality, cf. \cite[Theorem 11.22]{FeNo17} and the structure of the tensor $\mathbb S$ defined in \eqref{SDef} it follows that
 \begin{equation}\label{KornIn}
 	\int_{\Omega _R}|\nabla (U-u)|^2 \mathrm{d}x \leq c \int_{\Omega _R} \left(\mathbb{S}(\nabla  u)-\mathbb S(\nabla U)\right)\cdot \nabla(u-U)\mathrm{d}x.
 \end{equation}
We notice that the difference $(u-U)(t)$ can be understood as an element of $W^{1,2}(\eR^3)$ for a.e. $t\in(0,T)$ after an the extension by zero. Therefore \cite[Theorem 11.22 (i)]{FeNo17} implies that the constant $c$ in the latter inequality is independent of $R$. Moreover, one deduces by the H\"older and Young inequalities and \eqref{KornIn} similarly as in \cite{FeNeSu14} 
\begin{equation*}
	\nu \int_{\Omega_R} \mathbb{S}(\nabla U) \cdot \nabla (U- u) \mathrm{d}x \leq \frac{\nu}{16} \int_{\Omega_R} \left(\mathbb{S}(\nabla U) - \mathbb{S}(\nabla u)\right):(\nabla U -\nabla u)dx + c\nu \int_{\Omega_R} |\mathbb{S}(\nabla U)|^2 \mathrm{d}x.
\end{equation*}
Having the latter inequality at hand, we obtain
\begin{equation}\label{I2Est}
\begin{split}
|I_2|\leq& \nu\int_0^t\left(c_1\|\mathbb S(\nabla U)\|^2_{L^2(\Omega_R)}+c_2\|\nabla(U-u)\|^2_{L^2(\Omega_R)}\right)\\
\leq& c\nu \|U\|^2_{L^2(0,T;W^{2,2}(\Omega_R))}
+ \frac{\nu}{2}\int_0^t\int_{\Omega_R}\left(\mathbb S(\nabla U)-\mathbb S(\nabla u)\right)\cdot\nabla(U-u)\\
\leq& c\nu\left(\|v\|^2_{L^2(0,T;W^{2,2}(\eR^3))}+\|\nabla\Psi\|^2_{L^2(0,T;W^{2,2}(\eR^3))}+\|w_R\|^2_{L^2(0,T;W^{2,2}(\eR^3))}\right)\\&+ \frac{\nu}{2}\int_0^t\int_{\Omega_R}\left(\mathbb S(\nabla U)-\mathbb S(\nabla u)\right)\cdot\nabla(U-u)\\
\leq& c(D,T)\nu(1+R^{-2})+ \frac{\nu}{2}\int_0^t\int_{\Omega_R}\left(\mathbb S(\nabla U)-\mathbb S(\nabla u)\right)\cdot\nabla(U-u)
\end{split}
\end{equation}
for suitably chosen $c_2$ by Lemma \ref{Lem:Eul}, \eqref{AcEnEst} and \eqref{CorrEst}. We proceed with treating the term $I_3$. Expanding derivatives of $P'(r)$ one has 
\begin{equation}\label{I3Split}
\begin{split}
I_3=&\varepsilon^{-1}\int_0^\tau\int_{\Omega_R}(r-\rho)P''(r)\tder s=\int_0^\tau\int_{\Omega_R}sP''(r)\tder s+\int_0^\tau\int_{\Omega_R}\varepsilon^{-1}(\varrho-\rho)P''(r)\tder s\\
=&\int_0^\tau\int_{\Omega_R}s\left(P''(r)-P''(\varrho)\right)\tder s+\frac{p'(\varrho)}{2\varrho}\left[\int_{\Omega_R}s^2\right]_{t=0}^{t=\tau}+\int_0^\tau\int_{\Omega_R}\varepsilon^{-1}(\varrho-\rho)\left(P''(r)-P''(\varrho)\right)\tder s\\
&+\int_0^\tau\int_{\Omega_R}\varepsilon^{-1}(\varrho-\rho)P''(\varrho)\tder s\\
=&I_{3,a}+I_{3,b}+I_{3,c}+I_{3,d}.
\end{split}
\end{equation} 
Notice that $I_{3,b}$ cancels out $J_{4,d}$ from \eqref{J3Split} because of \eqref{EnConservRD}, \eqref{RadCond} and \eqref{ConstACW}. We estimate using \eqref{AcSys}$_1$
\begin{equation*}
\left|I_{3,a}\right|\leq  \int_0^\tau\int_\Omega \varepsilon s^2|\tder s| \overline P= \varrho\overline P\int_0^\tau\int_\Omega s^2|\Delta\Psi|\leq \varrho\overline P \|s\|_{L^2(0,T;L^2(\Omega))}\|s\|_{L^{q_1}(0,T;L^{q_1}(\Omega))}\|\Delta\Psi\|_{L^{q_2}(0,T;L^{q_2}(\Omega))}¨
\end{equation*}
where we denoted $\overline P=\max_{\{z\in[\varrho-\varepsilon_1\|s\|_{L^\infty(0,T;L^\infty(\eR^3))},\varrho+\varepsilon_1\|s\|_{L^\infty(0,T;L^\infty(\eR^3))}\}}|P'''(z)|$, where $\varepsilon_1$ is specified in \eqref{Eps1Def}. The exponents $q_1,q_2$ satisfy $\frac{1}{q_1}+\frac{1}{q_2}+\frac{1}{2}=1$. Hence using \eqref{AcEnEst} and \eqref{AcLpqEst} with any $q_1,q_2>2$ it follows that
\begin{equation}\label{I3AEst}
|I_{3,a}|\leq c(D,T)\varepsilon^{2(1-\frac{1}{q_1}-\frac{1}{q_2})}\leq c(D,T)\varepsilon.
\end{equation}
Similarly, we obtain
\begin{equation}\label{I3CEst}
\begin{split}
|I_{3,c}|\leq& \varrho\overline P\int_0^\tau\int_{\Omega_R}\left|\frac{\varrho-\rho}{\varepsilon}\right||s||\Delta\Psi|\\
\leq& c\left\|\frac{\varrho-\rho}{\varepsilon}\right\|_{L^2(0,T;L^2(\Omega))}\|s\|_{L^{q_1}(0,T;L^{q_1}(\Omega))}\|\Delta\Psi\|_{L^{q_2}(0,T;L^{q_2}(\Omega))}\leq c(D,T)\varepsilon^{2(1-\frac{1}{q_1}-\frac{1}{q_2})}
\end{split}
\end{equation}
Taking into account that
\begin{equation*}
\int_{\Omega_R}\nabla P'(r)\cdot rU=\int_{\Omega_R}p'(r)\nabla r\cdot U=-\int_{\Omega_R}p(r)\dvr U
\end{equation*}
by \eqref{PressPotId} we realize that 
\begin{equation}\label{IntSum}
I_4+I_5=-\varepsilon^{-2}\int_0^\tau \int_{\Omega_R}\nabla P'(r)\cdot\rho u-\varepsilon^{-2}\int_0^\tau \int_{\Omega_R}p(\rho)\dvr U=\tilde I_4+\tilde I_5.
\end{equation}

Applying \eqref{PressPotId} and equation \eqref{AcSys}$_2$ we get
\begin{align*}
\tilde I_4&=-\varepsilon^{-1}\int_0^\tau\int_{\Omega_R} P''(r)\nabla s\cdot\rho u\\
&=-\varepsilon^{-1}\int_0^\tau\int_{\Omega_R} \left(P''(r)-P''(\varrho)\right)\nabla s\cdot\rho u-\varepsilon^{-1}\int_0^\tau\int_{\Omega_R}\frac{p'(\varrho)}{\varrho}\nabla s\cdot\rho u\\
&=-\varepsilon^{-1}\int_0^\tau\int_{\Omega_R} \left(P''(r)-P''(\varrho)\right)\nabla s\cdot\rho u+\int_0^\tau\int_{\Omega_R}\rho u\cdot\tder\nabla\Psi=\tilde I_{4,a}+\tilde I_{4,b}.
\end{align*}
We proceed with estimates $\tilde I_{4,a}$. We note that the term $\tilde I_{4,b}$ cancels out its counterpart $J_{4,f}$ from \eqref{J3Split}. 
Similiarly to estimate \eqref{I3AEst}, we obtain
\begin{equation}\label{I4AEst}
|\tilde I_{4,a}|\leq \varrho^2\overline P\int_0^\tau\int_{\Omega_R}|s||\nabla s||\rho u|\leq c \|s\|_{L^{q_1}(0,T;L^{q_1}(\Omega))}\|\nabla s\|_{L^{q_2}(0,T;L^{q_2}(\Omega))}\|\rho u\|_{L^2(0,T;L^2(\Omega))}\leq c(D,T)\varepsilon^{2(1-\frac{1}{q_1}-\frac{1}{q_2})}.
\end{equation}
We continue with the estimate of $\tilde I_5$ from \eqref{IntSum}. We first rewrite it as
\begin{align*}
\tilde I_5&=\varepsilon^{-2}\int_0^\tau\int_{\Omega_R}\left(p(\rho)-p(\varrho)-p'(\varrho)(\rho-\varrho)\right)\dvr U+\varepsilon^{-2}\int_0^\tau\int_{\Omega_R}\left(p(\varrho)+p'(\varrho)(\rho-\varrho)\right)\dvr U\\
&=\tilde I_{5,a}+\tilde I_{5,b}
\end{align*}
By the definition of $U$ we obtain $\dvr U=\dvr(w_R+\nabla\Psi)$. Next we observe that thanks to estimates \eqref{CorrEst}, \eqref{AcEnEst} and the Sobolev embedding there is $\alpha_2$ such that
\begin{equation}\label{Alph2Def}
-\log(\overline \rho-s)\geq 8\|\dvr(w_R+\nabla\Psi)\|_{L^\infty((0,T)\times\eR^d)}\text{ if }s\in(\overline\rho-\alpha_2,\overline\rho).
\end{equation}
For the purposes of this subsection we define $b\in C^1([0,\overline\rho))$ with $b'\geq 0$ in the following way
\begin{equation}\label{bDef}
b(s)=\begin{cases}
0&\text{ if }s\leq\overline\rho-\alpha_1\\
-\log(\overline\rho-s)&\text{ if }\overline\rho-\alpha_2\leq s<\overline{\rho}
\end{cases}
\end{equation}
and $b'(s)>0$ for $s\in(\overline\rho-\alpha_1,\overline\rho-\alpha_2)$ with $\alpha_1$ from Lemma \ref{Lem:PointPPotEst}. We point out that such a function $b$ is admissible in Theorem \ref{Thm:Main} as assumption \eqref{PressConstr} implies that the conditions in \eqref{BNondec} are satisfied. 
Then we have
\begin{align*}
\tilde I_{5,a}=&\varepsilon^{-2}\int_{\{\rho\leq\overline\rho-\alpha_1\}}(p(\rho)-p(\varrho)-p'(\varrho)(\rho-\varrho)\left(\dvr w_R+\Delta\Psi\right)\\
&+\varepsilon^{-2}\int_{\{\rho\in(\overline\rho-\alpha_1,\overline\rho-\alpha_2)\}}(p(\rho)-p(\varrho)-p'(\varrho)(\rho-\varrho)\left(\dvr w_R+\Delta\Psi\right)\\
&+\varepsilon^{-2}\int_{\{\rho\geq\overline\rho-\alpha_2\}}(p(\rho)-p(\varrho)-p'(\varrho)(\rho-\varrho)\left(\dvr w_R+\Delta\Psi\right).
\end{align*}
Combining \eqref{PEst} with \eqref{PPotEst} we obtain, using also the definition of the function $b(\rho)$ in \eqref{bDef} and the definition of $\alpha_2$ in \eqref{Alph2Def},
\begin{equation}\label{I5AEst}
\begin{split}
|\tilde I_{5,a}|\leq&\varepsilon^{-2}\int_{\{\rho\leq\overline\rho-\alpha_1\}}\|\dvr(w_R+\nabla\Psi)\|_{L^\infty(\Omega_R)}\left(P(\rho)-P(\varrho)-P'(\varrho)(\rho-\varrho)\right)\\
&+\varepsilon^{-2}\int_{\{\rho\in(\overline\rho-\alpha_1,\overline\rho-\alpha_2)\}}\|\dvr(w_R+\nabla\Psi)\|_{L^\infty(\Omega_R)}\max_{z\in[\overline\rho-\alpha_1,\overline\rho-\alpha_2]}p''(z)(\rho-\varrho)^2\\&+2\varepsilon^{-2}\int_{\{\rho\geq\overline\rho-\alpha_2\}}p(\rho)\|\dvr(w_R+\nabla\Psi)\|_{L^\infty((0,T)\times\Omega_R)}\\
\leq& c\varepsilon^{-2}\left(\int_{\{\rho\leq\overline\rho-\alpha_1\}}(P(\rho)-P(\varrho)-P'(\varrho)(\rho-\varrho))+\int_{\{\rho\in(\overline\rho-\alpha_1,\overline\rho-\alpha_2)\}}(P(\rho)-P(\varrho)-P'(\varrho)(\rho-\varrho))\right)\\
&+\frac{1}{4\varepsilon^2}\int_{\{\rho\geq\overline\rho-\alpha_2\}}p(\rho)b(\rho)\\
\leq &c(D,T)(1+R^{-2})\int_0^\tau\mathcal E(\rho,u|r,U)(t)\dt+\frac{1}{4\varepsilon^2}\int_0^\tau\int_{\Omega_R}p(\rho)b(\rho).
\end{split}
\end{equation}
Let us handle the term $\tilde I_{5,b}$. Using the divergence theorem and the fact that $U$ possesses zero trace on $\partial\Omega_R$ we get 
\begin{equation*}
\begin{split}
\tilde I_{5,b}=& \varepsilon^{-2}p'(\varrho)\int_0^\tau\int_{\Omega_R}(\rho-\varrho)\dvr w_R+\varepsilon^{-2}p'(\varrho)\int_0^\tau\int_{\Omega_R}(\rho-\varrho)\Delta\Psi=\tilde J_{5,a}+\tilde J_{5,b}.
\end{split}
\end{equation*}
By \eqref{EpsInd}$_2$ and \eqref{CorrEst} we deduce 
\begin{equation}\label{J5AEst}
|\tilde J_{5,a}|\leq \varepsilon^{-1}\|\rho-\varrho\|_{L^\infty(0,T;L^2(\Omega_R))}\varepsilon^{-1}\|w_R\|_{L^1(0,T;W^{1,2}(\Omega_R))}\leq \frac{c(D,T)}{\varepsilon R}.
\end{equation}
Using \eqref{AcSys}$_2$ we deduce that $\tilde J_{5,b}$ cancels out $I_{3,d}$ from \eqref{I3Split}. 
Collecting estimates \eqref{J1Est}, \eqref{J2ABCEst}, \eqref{J2DEst}, \eqref{J3Est},\eqref{J567Est}, \eqref{I2Est}, \eqref{I3AEst},\eqref{I3CEst}, \eqref{I4AEst}, \eqref{I5AEst} and \eqref{J5AEst} we obtain 
\begin{equation}\label{R1Est}
\begin{split}
\int_0^\tau\mathcal R_1(t)\dt\leq &c(D,T)(1+R^{-2})\int_0^\tau\mathcal E(\rho,u|r,U)+\frac{1}{4}\varepsilon^{-2}\int_0^\tau\int_{\Omega_R}p(\rho)b(\rho)\\&+c(D,T)\left(\varepsilon^{\alpha}+R^{-1}+(\varepsilon R)^{-1}+\nu\right)\\
&+\frac{\nu}{2}\int_0^\tau\int_{\Omega_R}\mathbb S(\nabla(u-U))\cdot\nabla(u-U)
\end{split}
\end{equation}
for any $\alpha\in(0,1)$.
For estimates of the terms in $\mathcal{R}_2$ we need some preparations. First, it follows from assumption \eqref{PressConstr} and \eqref{PresPot} that for any $\gamma>0$
\begin{equation}
\lim_{s\to\overline\rho_-}\frac{p(s)}{(b(s))^\gamma}=\lim_{s\to\overline\rho_-}\frac{P(s)}{(b(s))^\gamma}=\lim_{s\to\overline\rho_-}\frac{p(s)}{(b'(s))^\beta}=\lim_{s\to\overline\rho_-}\frac{P(s)}{(b(s))^{\beta-1}}=+\infty.
\end{equation}
The latter results imply that for $\gamma\geq 1$
\begin{equation}\label{BPressEst}
\int_{\Omega_R} |b(\rho)|^\gamma=\int_{\{\rho>\overline\rho-\alpha_1\}}|b(\rho)|^\gamma\leq c\int_{\{\rho>\overline\rho-\alpha_1\}}P(\rho)\leq c \int_{\{\rho>\overline\rho-\alpha_1\}}\left(P(\rho)-P(r)-P'(r)(\rho-r)\right)
\end{equation}
with $\alpha_1$ from Lemma \ref{Lem:PointPPotEst}. 
Moreover, for any $\beta_0\in[2,\beta]$ we have
\begin{equation}
\begin{split}
\int_{\Omega_R}|b'(\rho)|^{\beta_0-1}\leq& c\int_{\{\rho>\overline\rho-\alpha_1\}}P(\rho)\leq c\int_{\{\rho>\overline\rho-\alpha_1\}}\left(P(\rho)-P(r)-P'(r)(\rho-r)\right),\\
\int_{\Omega_R}|b'(\rho)|^{\beta_0}\leq& c\int_{\Omega_R}p(\rho).
\end{split}
\end{equation}
Using \eqref{BPressEst} with $\gamma=1$ we conclude
\begin{equation}\label{PresMeanValEst}
\begin{split}
I_6\leq& c\int_0^\tau\frac{1}{|\Omega_R|}\int_{\Omega_R}p(\rho)\varepsilon^{-2}\int_{\Omega_R}\left(P(\rho)-P(r)-P'(r)(\rho-r)\right)\\
\leq& c\int_0^\tau\frac{1}{|\Omega_R|}\int_{\Omega_R}p(\rho) \mathcal E(\rho, u|r,U)(t)
\end{split}
\end{equation}
In order to treat the second term in $\mathcal R_2$ we adopt the computations from \cite[Section 4.5]{FeLuNo18} in the following way. First, we write
\begin{align*}
 I_7=&
-\int_0^\tau\int_{\Omega_R}\rho(u-U)\otimes(u-U)\cdot\nabla\mathcal B \left(b(\rho)-\frac{1}{|\Omega_R|}\int_{\Omega_R}b(\rho)\right)\\
&-\int_0^\tau\int_{\Omega_R}\rho U\otimes (u-U)\cdot\nabla \mathcal B \left(b(\rho)-\frac{1}{|\Omega_R|}\int_{\Omega_R}b(\rho)\right)\\
&-\int_0^\tau\int_{\Omega_R}\rho(u-U)\otimes U\cdot \nabla \mathcal B \left(b(\rho)-\frac{1}{|\Omega_R|}\int_{\Omega_R}b(\rho)\right)\\
&-\int_0^\tau\int_{\Omega_R}(\rho-r)U\otimes U\cdot \nabla \mathcal B \left(b(\rho)-\frac{1}{|\Omega_R|}\int_{\Omega_R}b(\rho)\right)\\
&-\int_0^\tau\int_{\Omega_R} rU\otimes U\cdot \nabla \mathcal B \left(b(\rho)-\frac{1}{|\Omega_R|}\int_{\Omega_R}b(\rho)\right)\\
=&I_{7,a}+I_{7,b}+I_{7,c}+I_{7,d}+I_{7,e}.
\end{align*}
Using the H\"older and Young inequalities, the Sobolev embedding, \eqref{KornIn}  and Lemma \ref{Lem:Bog} we obtain
\begin{align*}
|I_{7,a}|\leq& \int_0^\tau\|\rho(u-U)\|_{L^2(\Omega_R)}\|u-U\|_{L^6(\Omega_R)}\left\|\nabla \mathcal B \left(b(\rho)-\frac{1}{|\Omega_R|}\int_{\Omega_R}b(\rho)\right)\right\|_{L^3(\Omega_R)}\\
\leq&c\nu^{-1}\int_0^\tau\|\sqrt\rho(u-U)\|^2_{L^2(\Omega_R)}\|b(\rho)\|^2_{L^3(\Omega_R)}+\frac{\nu}{16}\int_0^\tau\int_{\Omega_R}\mathbb S(\nabla(u-U))\cdot\nabla(u-U).
\end{align*}
Hence using \eqref{BPressEst} and \eqref{PPotEpsEst} we conclude
\begin{equation}\label{I7AEst}
\begin{split}
|I_{7,a}|\leq& c \nu^{-1}\int_0^\tau \left(\int_{\Omega_R} P(\rho)\right)^\frac{2}{3}\mathcal E(\rho, u|r, U)+\frac{\nu}{16}\int_0^\tau\int_{\Omega_R}\mathbb S(\nabla(u-U))\cdot\nabla(u-U)\\
\leq& c\varepsilon^\frac{4}{3}\nu^{-1}\int_0^\tau \mathcal E(\rho,u|r,U)+\frac{\nu}{16}\int_0^\tau\int_{\Omega_R}\mathbb S(\nabla(u-U))\cdot\nabla(u-U).
\end{split}
\end{equation}
Next, \eqref{BPressEst}, \eqref{PPotEpsEst} and Lemma \ref{Lem:Bog} imply
\begin{equation}\label{I7BCEst}
\begin{split}
|I_{7,b}|+|I_{7,c}|\leq& \overline\rho\int_0^\tau\|U\|_{L^\infty(\Omega_R)}\|\sqrt\rho(u-U)\|^2_{L^2(\Omega_R)}+\int_0^\tau\|U\|_{L^\infty(\Omega_R)}\left\|\nabla \mathcal B\left(b(\rho)-\frac{1}{|\Omega_R|}\int_{\Omega_R}b(\rho)\right)\right\|^2_{L^2(\Omega_R)}\\
\leq&c\int_0^\tau \|U\|_{L^\infty(\Omega_R)}(\|\sqrt\rho(u-U)\|^2_{L^2(\Omega_R)}+\|b(\rho)\|^2_{L^2(\Omega_R)})\\
\leq &c \int_0^\tau \|U\|_{L^\infty(\Omega_R)}\left(\|\sqrt\rho(u-U)\|^2_{L^2(\Omega_R)}+\int_{\Omega_R}(P(\rho)-P(r)-P'(r)(\rho-r))\right)\\
\leq &c(D,T)(1+\varepsilon^2)(1+R^{-2}) \int_0^\tau\mathcal E(\rho,u|r,U).
\end{split}
\end{equation}
Applying the Young inequality and \eqref{PPotEst} it follows that
\begin{equation}\label{I7DEst}
\begin{split}
|I_{7,d}|\leq& \int_0^\tau \|U\|^2_{L^\infty(\Omega)}\left(\|\rho-r\|^2_{L^2(\Omega_R)}+\left\|\nabla \mathcal B\left(b(\rho)-\frac{1}{|\Omega_R|}\int_{\Omega_R}b(\rho)\right)\right\|^2_{L^2(\Omega_R)}\right)\\
\leq &c \int_0^\tau \|U\|^2_{L^\infty(\Omega_R)}\left(\int_{\Omega_R}(P(\rho)-P(r)-P'(r)(\rho-r))+\|b(\rho)\|^2_{L^2(\Omega_R)}\right)\\
\leq &c(D,T)\varepsilon^2(1+R^{-4})\int_0^\tau\mathcal E(\rho,u|r,U).
\end{split}
\end{equation}
Then we estimate by the Sobolev embedding and Lemma \ref{Lem:Bog}
\begin{equation}\label{I7EEst}
\begin{split}
|I_{7,e}|\leq &c\int_0^\tau\|r\|_{L^\infty(\Omega_R)}\|U\|^2_{L^4(\Omega_R)}\|b(\rho)\|_{L^2(\Omega_R)}\\
\leq &c\varepsilon^2\int_0^\tau (\varrho+\varepsilon_0\|s\|_{L^\infty(\Omega_R)})^2(\|v\|_{L^4(\Omega_R)}+\|\nabla \Psi\|_{L^4(\Omega_R)}+\|w_R\|_{L^4(\Omega_R)})^4+c\varepsilon^{-2}\int_0^\tau\|b(\rho)\|^2_{L^2(\Omega_R)}\\ \leq &c(D,T,\varepsilon_0)\varepsilon^2+c\int_0^\tau\mathcal E(\rho,u|r,U).
\end{split}
\end{equation}
In order to handle the next term of $\mathcal R_2$ we write 
\begin{equation*}
\begin{split}
I_8
=&\nu\int_0^\tau\int_{\Omega_R}\left(\mathbb S(\nabla u)-\mathbb S(\nabla U)\right)\cdot\nabla\mathcal B\left(b(\rho)-\frac{1}{|\Omega_R|}\int_{\Omega_R}b(\rho)\right)\\
&+\nu\int_0^\tau\int_{\Omega_R}\mathbb S(\nabla U)\cdot\nabla\mathcal B\left(b(\rho)-\frac{1}{|\Omega_R|}\int_{\Omega_R}b(\rho)\right)\\
=&I_{8,a}+I_{8,b}.
\end{split}
\end{equation*}
Using the H\"older and Young inequalities, the structure of the tensor $\mathbb{S}(\nabla u)$, Lemma \ref{Lem:Bog} and \eqref{PPotEst} we deduce
\begin{equation}\label{I8AEst}
\begin{split}
|I_{8,a}&|\leq \nu\int_0^\tau\|\left(\mathbb S(\nabla u)-\mathbb S(\nabla U)\right)\|_{L^2(\Omega_R)}\left\|\nabla \mathcal B\left(b(\rho)-\frac{1}{|\Omega_R|}\int_{\Omega_R}b(\rho)\right)\right\|_{L^2(\Omega_R)}\\
&\leq \frac{\nu}{16}\int_0^\tau\int_{\Omega_R}\left(\mathbb S(\nabla u)-\mathbb S(\nabla U)\right)\cdot\nabla(u-U)+c\nu\int_0^\tau\|b(\rho)\|^2_{L^2(\Omega_R)}\\
&\leq \frac{\nu}{16}\int_0^\tau\int_{\Omega_R}\left(\mathbb S(\nabla u)-\mathbb S(\nabla U)\right)\cdot\nabla(u-U)+c\varepsilon^{2}\nu\int_0^\tau \mathcal E(\rho,u|r,U).
\end{split}
\end{equation}
Folowing the arguments used in \eqref{I7EEst} we get 
\begin{equation}\label{I8BEst}
\begin{split}
|I_{8,b}|=&\left|\nu\int_0^\tau\int_{\Omega_R}\dvr \mathbb S(\nabla U)\cdot\mathcal B\left(b(\rho)-\frac{1}{|\Omega_R|}\int_{\Omega_R}b(\rho)\right)\right|\leq c\nu\int_0^\tau\|U\|_{W^{2,2}(\Omega_R)}\|b(\rho)\|_{L^2(\Omega_R)}\\
\leq& c(D,T)\varepsilon^2\nu^2(1+R^{-2})+c\varepsilon^{-2}\int_0^\tau\|b(\rho)\|^2_{L^2(\Omega_R)}
\leq c(D,T)\varepsilon^2\nu^2(1+R^{-2})+c\int_0^\tau\mathcal E(\rho,u|r,U).
\end{split}
\end{equation}
Next, we write
\begin{equation*}
\begin{split}
I_9=&\int_0^\tau\int_{\Omega_R} \rho u\cdot \mathcal B(\dvr(b(\rho)(u-U)))\\
&+\int_0^\tau\int_{\Omega_R} \rho(u-U)\cdot\mathcal B(\dvr(b(\rho)U))\\
&+\int_0^\tau\int_{\Omega_R} (\rho-r)U\cdot\mathcal B(\dvr(b(\rho)U))\\
&+\int_0^\tau\int_{\Omega_R} rU\cdot\mathcal B(\dvr(b(\rho)U))=I_{9,a}+I_{9,b}+I_{9,c}+I_{9,d}.
\end{split}
\end{equation*}
By the H\"older and Young inequalities, Lemma \ref{Lem:Bog}, the fact that $\|\dvr(b(\rho)(u-U))\|_{(\dot{W}^{1,\frac{3}{2}(\Omega)})'}\leq \|b(\rho)(u-U)\|_{L^\frac{3}{2}(\Omega)}$, \eqref{BPressEst}, \eqref{PPotEpsEst}, the Sobolev embedding and the Korn inequality it follows that
\begin{equation}\label{I9AEst}
\begin{split}
|I_{9,a}|\leq& c\int_0^\tau\|\rho u\|_{L^2(\Omega_R)}\|b(\rho)(u-U)\|_{L^2(\Omega_R)}\\
\leq &c\nu^{-1}\|\rho u\|^2_{L^\infty(0,T;L^2(\Omega_R))}\int_0^\tau \|b(\rho)\|^2_{L^4(\Omega_R)}+\sigma\nu\int_0^\tau\|u-U\|^2_{L^4(\Omega_R)}\\
\leq& c\nu^{-1}\int_0^\tau\left(\int_{\{\rho(t,\cdot)>\overline\rho-\alpha\}}P(\rho)\right)^\frac{1}{2}+\frac{\nu}{16}\int_0^\tau\int_{\Omega_R}\left(\mathbb S(\nabla u)-\mathbb S(\nabla U)\right)\cdot\nabla(u-U)\\
\leq& c \varepsilon\nu^{-1}+\frac{\nu}{16}\int_0^\tau\int_{\Omega_R}\left(\mathbb S(\nabla u)-\mathbb S(\nabla U)\right)\cdot\nabla(u-U).
\end{split}
\end{equation}
We similarly deduce 
\begin{equation}\label{I9BEst}
\begin{split}
|I_{9,b}|\leq & \overline\rho\int_0^\tau\int_{\Omega_R}\rho|u-U|^2+c\int_0^\tau\|b(\rho)U\|^2_{L^2(\Omega_R)}\\ \leq& c \int_0^\tau\mathcal E(\rho,u|r,U)+c\varepsilon^2\int_0^\tau\|U\|^2_{L^\infty(\Omega_R)}\varepsilon^{-2}\int_{\Omega_R}(P(\rho)-P(r)-P'(r)(\rho-r))\\
\leq & c(D,T)\int_0^\tau (1+\varepsilon^2(1+R^{-2}))\mathcal E(\rho,u|r,U),
\end{split}
\end{equation}
\begin{equation}\label{I9CEst}
|I_{9,c}|\leq \int_0^\tau\|U\|^2_{L^\infty(\Omega_R)}(\rho-r)^2+c\int_0^\tau\|U\|^2_{L^\infty(\Omega_R)}\|b(\rho)\|^2_{L^2(\Omega_R)}
\leq c(1+\varepsilon^2)(1+R^{-2})\int_0^\tau\mathcal E(\rho,u|r,U)
\end{equation}
and
\begin{equation}\label{I9DEst}
\begin{split}
|I_{9,d}|\leq& c\int_0^\tau \|rU\|_{L^2(\Omega_R)}\|b(\rho)U\|_{L^2(\Omega_R)}\leq c\varepsilon^2\int_0^\tau \|r\|^2_{L^\infty(\Omega_R)}\|U\|^2_{L^2(\Omega_R)}+c\varepsilon^{-2}\int_0^\tau\|U\|^2_{L^\infty(\Omega_R)}\|b(\rho)\|^2_{L^2(\Omega_R)}\\
\leq& c(D,T)\varepsilon^2(1+R^{-4})\int_0^\tau(\varrho+\varepsilon_0\|s\|_{L^\infty(\Omega_R)})^2 +c(D,T)(1+R^{-2})\int_0^\tau\mathcal E(\rho,u|r,U).
\end{split}
\end{equation}
Next, we write
\begin{equation*}
\begin{split}
I_{10}=&\int_0^\tau\int_{\Omega_R} \rho u\cdot \mathcal B\left((b'(\rho)\rho-b(\rho))\dvr U-\frac{1}{|\Omega_R|}\int_{\Omega_R}(b'(\rho)\rho-b(\rho))\dvr U\right)\\
&+\int_0^\tau\int_{\Omega_R} \rho U\cdot \mathcal B\left((b'(\rho)\rho-b(\rho))\dvr (u-U)-\frac{1}{|\Omega_R|}\int_{\Omega_R}(b'(\rho)\rho-b(\rho))\dvr (u-U)\right)\\
&+\int_0^\tau\int_{\Omega_R} \rho (u-U)\cdot \mathcal B\left((b'(\rho)\rho-b(\rho))\dvr (u-U)-\frac{1}{|\Omega_R|}\int_{\Omega_R}(b'(\rho)\rho-b(\rho))\dvr (u-U)\right)\\
=&I_{10,a}+I_{10,b}+I_{10,c}.
\end{split}
\end{equation*}
By Lemma \ref{Lem:Bog}, \eqref{EpsInd}$_1$ and \eqref{BPressEst} it follows that 
\begin{equation}\label{I10AEst}
\begin{split}
|I_{10,a}|\leq& c\int_0^\tau\|\rho u\|_{L^2(\Omega_R)}\|(b'(\rho)\rho-b(\rho))\dvr U\|_{L^2(\Omega_R)}\\
\leq& c\|\sqrt\rho u\|_{L^\infty(0,T;L^2(\Omega_R))}\int_0^\tau \|U\|_{W^{1,\infty}(\Omega_R)}\left(\|b'(\rho)\|_{L^2(\Omega_R)}+\|b(\rho)\|_{L^2(\Omega_R)}\right)\\
\leq& c\varepsilon^2\int_0^\tau\|U\|^2_{W^{1,\infty}(\Omega_R)}+\varepsilon^{-2}\int_0^\tau\left(\|b'(\rho)\|^2_{L^2(\Omega_R)}+\|b(\rho)\|^2_{L^2(\Omega_R)}\right)\\
\leq& c(D,T)\varepsilon^2(1+R^{-4})+c\int_0^\tau\mathcal E(\rho,u|r,U).
\end{split}
\end{equation}
The Sobolev embedding $W^{1,\frac{6}{5}}_0(\Omega_R)$ into $L^1(\Omega_R)$, Lemma \ref{Lem:Bog}, \eqref{BPressEst} and \eqref{PPotEpsEst} yield
\begin{equation}\label{I10BEst}
\begin{split}
|I_{10,b}|\leq& \overline\rho\int_0^\tau\|U\|_{L^\infty(\Omega_R)}\left\|\mathcal B\left((b'(\rho)\rho-b(\rho))\dvr(u-U)-\frac{1}{|\Omega_R|}\int_{\Omega_R}(b'(\rho)\rho-b(\rho))\dvr(u-U)\right)\right\|_{W^{1,\frac{6}{5}}(\Omega_R)}\\
\leq& c\int_0^\tau\|U\|_{L^\infty(\Omega_R)}\|(b'(\rho)\rho-b(\rho))\dvr(u-U)\|_{L^\frac{6}{5}(\Omega_R)}\\
\leq& c\int_0^\tau\|U\|_{L^\infty(\Omega_R)}\|(b'(\rho)\rho-b(\rho))\|_{L^3(\Omega_R)}\|\nabla(u-U)\|_{L^2(\Omega_R)}\\
\leq& c\nu^{-1}\int_0^\tau\|U\|^2_{L^\infty(\Omega_R)}(\|b'(\rho)\|^2_{L^3(\Omega_R)}+\|b(\rho)\|^2_{L^3(\Omega_R)})+\frac{\nu}{16}\int_0^\tau\int_{\Omega_R}\mathbb S(\nabla(u-U))\cdot\nabla(u-U)\\ 
\leq& c\nu^{-1}\int_0^\tau\|U\|^2_{L^\infty(\Omega_R)}\left(\int_{\{\rho(t,\cdot)>\overline\rho-\alpha_1\}}P(\rho)\right)^\frac{2}{3}+\frac{\nu}{16}\int_0^\tau\int_{\Omega_R}\left(\mathbb S(\nabla u)-\mathbb S(\nabla U)\right)\cdot\nabla(u-U)\\ 
\leq& c(D,T)\varepsilon^\frac{4}{3}\nu^{-1}(1+R^{-4})+\frac{\nu}{16}\int_0^\tau\int_{\Omega_R}\left(\mathbb S(\nabla u)-\mathbb S(\nabla U)\right)\cdot\nabla(u-U)
\end{split}
\end{equation}
as well as
\begin{equation}\label{I10CEst}
\begin{split}
|I_{10,c}|\leq& c\int_0^\tau \|\sqrt\rho(u-U)\|_{L^2(\Omega_R)}\|\mathcal B((b'(\rho)\rho-b(\rho))\dvr(u-U))\|_{W^{1,\frac{6}{5}}(\Omega_R)}\\
\leq&c\int_0^\tau \|\sqrt\rho(u-U)\|_{L^2(\Omega_R)}\|(b'(\rho)\rho-b(\rho))\dvr(u-U)\|_{L^{\frac{6}{5}}(\Omega_R)}\\
\leq& c\nu^{-1}\int_0^\tau  \|\sqrt\rho(u-U)\|^2_{L^2(\Omega_R)}\left(\|b'(\rho)\|^2_{L^3(\Omega_R)}+\|b'(\rho)\|^2_{L^3(\Omega_R)}\right)+c\nu\int_0^t\|\nabla(u-U)\|^2_{L^2(\Omega_R)}\\
\leq&c\nu^{-1}\int_0^\tau \|\sqrt\rho(u-U)\|^2_{L^2(\Omega_R)}\left(\int_{\{\rho(t,\cdot)>\overline\rho-\alpha_1\}}P(\rho)\right)^\frac{2}{3}+\frac{\nu}{16}
\int_0^\tau\int_{\Omega_R}\left(\mathbb S(\nabla u)-\mathbb S(\nabla U)\right)\cdot\nabla(u-U)\\
\leq&c\varepsilon^\frac{4}{3}\nu^{-1}\int_0^\tau \mathcal E(\rho,u|r,U)+\frac{\nu}{16}
\int_0^\tau\int_{\Omega_R}\left(\mathbb S(\nabla u)-\mathbb S(\nabla U)\right)\cdot\nabla(u-U).
\end{split}
\end{equation}
Collecting estimates \eqref{PresMeanValEst}--\eqref{I10CEst} we conclude 
\begin{equation}\label{R2Est}
\begin{split}
\int_0^\tau&\mathcal R_2(t)\dt\\
\leq& c(D,T)\int_0^\tau \Bigl(\frac{1}{|\Omega_R|}\int_{\Omega_R}p(\rho)+\varepsilon^\frac{4}{3}\nu^{-1}(1+R^{-4})+(1+\varepsilon^2)(1+R^{-2})
+\varepsilon^2\nu+1\Bigr)\mathcal E(\rho,u|r,U)(t)\dt\\
&+\frac{5\nu}{16}\int_0^\tau\int_{\Omega_R}\left(\mathbb S(\nabla u)-\mathbb S(\nabla U)\right)\cdot\nabla(u-U)+c(D,T)(\varepsilon^2+\varepsilon^2\nu^2(1+R^{-2})+\varepsilon\nu^{-1}).
\end{split}
\end{equation}
Similarly as berfore we deduce
\begin{equation}\label{R3Est}
\begin{split}
\mathcal{R}_3(t)\leq& c\|\rho u\|_{L^\infty(0,T;L^2(\Omega_R))}\|b(\rho)(t,\cdot)\|_{L^2(\Omega_R)}+c\|\rho_0u_0\|_{L^2(\Omega_R)}\|b(\rho_0)\|_{L^2(\Omega_R)}\\
\leq& c(D,T)\varepsilon^2+\sigma\varepsilon^{-2}\int_{\{\rho(t,\cdot)>\overline\rho-\alpha_1\}}|b(\rho)|^2+\sigma\varepsilon^{-2}\int_{\{\rho_0(\cdot)>\overline\rho-\alpha_1\}}|b(\rho_0)|^2\\
\leq&  c(D,T)¨\varepsilon^2+\frac{1}{4}\mathcal E(\rho,u|r,U)(t)+\frac{1}{4}\mathcal E(\rho_0,u_0|r(0,\cdot),U(0,\cdot)),
\end{split}
\end{equation}
where $\sigma$ was suitably chosen and \eqref{EpsInd}$_1$ along with assumption \eqref{InitDBound} were taken into account. Notice that the generic constants appeaing in estimates \eqref{R2Est} and \eqref{R3Est} contain also positive powers of the expression $\frac{\diam(\Omega_R)}{R}$ that is bounded with respect to $R$ due to assumption \eqref{BoundOmLimits}.
Going back to \eqref{REIEps}, which we combine with estimates \eqref{R1Est}, \eqref{R2Est} and \eqref{R3Est}, it follows by the Gronwall lemma that
\begin{equation}\label{GRIneq}
\mathcal E(\rho,u|r,U)(t)\leq \gamma\exp\left(\int_0^t\delta(s)\ds\right)\ t\in[0,T],
\end{equation}
where 
\begin{align*}
\gamma=&c\mathcal E(\rho_0,u_0|r(0,\cdot),U(0,\cdot))+c(D,T)(\varepsilon^\alpha+R^{-1}+(\varepsilon R)^{-1}+\nu+\varepsilon^2\nu^2(1+R^{-2})+\varepsilon\nu^{-1}),\\
\delta(\tau)=& c(D,T)\Bigl(\frac{1}{|\Omega_R|}\int_{\Omega_R}p(\rho)+\varepsilon^\frac{4}{3}\nu^{-1}(1+R^{-4})+(1+\varepsilon^2)(1+R^{-2})+\varepsilon^2\nu+1\Bigr).
\end{align*}
In order to proceed we estimate the quantity $|\Omega_R|^{-1}\int_0^\tau\int_{\Omega_R}p(\rho)$.
We begin by repeating the part of the proof of Theorem \ref{Thm:Main}, namely the proof of \eqref{PressRenId}, as $b(s)=s$ satisfies the assumptions of Theorem \ref{Thm:Main}, to get 
\begin{equation}\label{BogRhoTest}
\varepsilon^{-2}\int_0^\tau\int_{\Omega_R}p(\rho)\left(\rho-\frac{1}{|\Omega_R|}\int_{\Omega_R}\rho\right)=\sum_{i=1}^4 I_i,
\end{equation}
where 
\begin{align*}
I_1&=-\int_0^\tau\int_{\Omega_R}\rho u\otimes u\cdot\nabla \mathcal B\left(\rho-\frac{1}{|\Omega_R|}\int_{\Omega_R}\rho\right),\\
I_2&=\int_0^\tau\int_{\Omega_R}\nu\mathbb{S}(\nabla u)\cdot\nabla\mathcal B\left(\rho-\frac{1}{|\Omega_R|}\int_{\Omega_R}\rho\right),\\
I_3&=\int_0^\tau\int_{\Omega_R}\rho u\cdot \mathcal B\left(\dvr(\rho u)-\frac{1}{|\Omega_R|}\int_{\Omega_R}\dvr(\rho u)\right),\\
I_4&=\int_{\Omega_R}\rho u\cdot\mathcal B \left(\rho-\frac{1}{|\Omega_R|}\int_{\Omega_R}\rho\right)-\int_{\Omega_R}\rho_0u_0\mathcal B\left(\rho_0-\frac{1}{|\Omega_R|}\int_{\Omega_R}\rho_0\right).
\end{align*}
By Lemma \ref{Lem:Bog} we conclude 
\begin{equation}\label{BogDens}
\left\|\mathcal B\left(\rho-\frac{1}{|\Omega_R|}\int_{\Omega_R}\rho\right)\right\|_{L^p(0,T;W^{1,q}(\Omega_R))}\leq c\|\rho\|_{L^p(0,T;L^q(\Omega_R))}\text{ for } p\in[1,\infty], q\in(1,\infty)
\end{equation}
with the constant $c$ dependentr also on the term of the form $\frac{\diam( \Omega_R)}{R}$ that is bounded with respect to $R$ due to the assumption \eqref{BoundOmLimits}. Using the fact that $\rho\leq\overline\rho$, \eqref{EpsInd}$_{1,3}$, the Sobolev embedding, \eqref{BogDens} and Lemma \ref{Lem:Bog}  we conclude
\begin{align*}
|I_1|&\leq c \overline\rho\|u\|^2_{L^2(0,T;L^6(\Omega_R))}\|\rho\|_{L^\infty(0,T;L^\frac{3}{2}(\Omega_R))}\leq c\nu^{-1}|\Omega_R|^\frac{2}{3},\\
|I_2|&\leq c\nu\|\mathbb S(\nabla u)\|_{L^2(0,T;L^2(\Omega_R))}\|\rho\|_{L^2(0,T;L^2(\Omega_R))}\leq c\nu^\frac{1}{2}T^\frac{1}{2}|\Omega_R|^\frac{1}{2},\\
|I_3|&\leq c\|\rho u\|^2_{L^2(0,T;L^2(\Omega_R))}\leq c,\\
|I_4|&\leq c\|\rho u\|_{L^\infty(0,T;L^2(\Omega_R))}\|\rho\|_{L^\infty(0,T;L^2(\Omega_R))}+c\|\rho_0 u_0\|_{L^2(\Omega_R)}\|\rho_0\|_{L^2(\Omega_R)}\leq c|\Omega_R|^\frac{1}{2}.
\end{align*}
Therefore, we get from \eqref{BogRhoTest} as $|\Omega_R|\geq cR^{3}$ by \eqref{BoundOmLimits}
\begin{equation}
\frac{1}{|\Omega_R|}\int_0^\tau\int_{\Omega_R}p(\rho)\left(\rho-\frac{1}{|\Omega_R|}\int_{\Omega_R}\rho\right)\leq c(D,T)\varepsilon^{2}(\nu^{-1}R^{-1}+c\nu^\frac{1}{2}R^{-\frac{3}{2}}+1).
\end{equation}
Next, we denote
\begin{equation*}
m_{\varepsilon,R}=\frac{1}{|\Omega_R|}\int_{\Omega_R}\rho_{0,\varepsilon}=\varrho+\frac{\varepsilon}{|\Omega_R|}\int_{\Omega_R}\rho^{(1)}_{0,\varepsilon}.
\end{equation*} By \eqref{EpsZCh} it follows that for any $\varepsilon<\varepsilon_0$
\begin{equation}
\frac{1}{|\Omega_R|}\int_{\Omega_R} \rho(t,x)\dx=m_{\varepsilon,R}<\overline\rho.
\end{equation}
Therefore we have 
\begin{equation}
\begin{split}
\frac{1}{|\Omega_R|}\int_0^\tau\int_{\Omega_R}p(\rho)\left(\rho-\frac{1}{|\Omega_R|}\int_{\Omega_R}\rho\right)=&\frac{1}{|\Omega_R|}\int_{\left\{\rho\leq\frac{\overline\rho+m_{\varepsilon,R}}{2}\right\}} p(\rho)\left(\rho-\frac{1}{|\Omega_R|}\int_{\Omega_R}\rho\right)\\
&¨+\frac{1}{|\Omega_R|}\int_{\left\{\rho>\frac{\overline\rho+m_{\varepsilon,R}}{2}\right\}} p(\rho)\left(\rho-\frac{1}{|\Omega_R|}\int_{\Omega_R}\rho\right)=J_1+J_2.
\end{split}
\end{equation}
Next, we conclude
\begin{equation}
J_1\leq \max_{s\in\left[0,\frac{1}{2}(\overline\rho+\varrho+\varepsilon_0D)\right]}p(s)\frac{\overline\rho-m_{\varepsilon,R}}{2}
\end{equation}
by the assumed continuity of $p$. On the other hand we get
\begin{equation}
J_2\geq \frac{\overline\rho-m_{\varepsilon,R}}{2|\Omega_R|}\int_{\left\{\rho>\frac{\overline\rho+m_{\varepsilon,R}}{2}\right\}}p(\rho)\geq \frac{\overline\rho-m_{\varepsilon,R}}{2|\Omega_R|}\int_{\left\{\rho>\frac{1}{2}(\overline\rho+\varrho+\varepsilon_0D)\right\}}p(\rho).
\end{equation}
Since $\inf_{\varepsilon<\varepsilon_0}(\overline\rho-m_{\varepsilon,R})\geq\overline\rho-\varrho-\varepsilon_0 D>0$ by \eqref{EpsZCh}, it follows  that 
\begin{align*}
\frac{1}{|\Omega_R|}\int_0^\tau\int_{\Omega_R}p(\rho)&=\frac{1}{|\Omega_R|}\int_{\left\{\rho\leq\frac{\overline\rho+\varrho+\varepsilon_0D}{2}\right\}}p(\rho)+\frac{1}{|\Omega_R|}\int_{\left\{\rho>\frac{\overline\rho+\varrho+\varepsilon_0D}{2}\right\}}p(\rho)\\
&\leq \max_{s\in[0,\frac{\overline\rho+\varrho+\varepsilon_0D}{2}]}p(s)+\frac{2}{\overline\rho-\varrho-\varepsilon_0 D}J_2\\
&\leq \max_{s\in[0,\frac{\overline\rho+\varrho+\varepsilon_0D}{2}]}p(s)+ \frac{2}{\overline\rho-\varrho-\varepsilon_0 D}\left(\frac{1}{|\Omega_R|}\left|\int_0^\tau\int_{\Omega_R} p(\rho)\left(\rho-\frac{1}{|\Omega_R|}\rho)\right)\right|+|J_1|\right)\\
&\leq\frac{ c\varepsilon^2}{\overline\rho-\varrho-\varepsilon_0 D}(\nu^{-1}R^{-1}+c\nu^\frac{1}{2}R^{-\frac{3}{2}}+1)+c\max_{s\in[0,\frac{\overline\rho+\varrho+\varepsilon_0D}{2}]}p(s).
\end{align*}
Employing the latter inequality in the $\delta$--term of \eqref{GRIneq} we conclude \eqref{FinIneq} by Lemma \ref{Lem:PressPotEst} and the proof is finished.

\section{Appendix}
The ensuing lemma deals with renormalized solutions of the continuity equation. It collects versions of assertions \cite[Lemmas 6.9 and 6.11]{NovStr04} adopted for a function $b$ considered in this paper.
\begin{Lemma}
	Let $T>0$ and $\Omega\subset\eR^d$, $d\geq 2$, be a bounded Lipschitz domain. Let $\rho\in L^\infty(Q_T)$ be such that $0\leq \rho<\bar\rho$ a.e. in $Q_T$ and $\rho$ togrether with $u\in L^2(W^{1,2}_0(\Omega)^d)$ satisfy the continuity equation
	\begin{equation}\label{CEq}
	\tder \rho+\dvr(\rho u)=0\text{ in }\mathcal D'(Q_T)
	\end{equation}
	Let $b\in C^1([0,\bar\rho))$ be such that
	\begin{equation}\label{BRhoReg}
	b(\rho)\in L^{2}(Q_T), b'(\rho)\in L^2(Q_T)
	\end{equation}
	and 
	\begin{equation}\label{BMon}
	b, b'\text{ are nondecreasing on }[\bar\rho-\alpha_0,\bar\rho)\text{ for some }\alpha_0\in(0,\bar\rho).
	\end{equation}
	Let $\rho$ and $u$ be extended by zero in $(0,T)\times\left(\eR^d\setminus\Omega\right)$.
	\begin{enumerate}
		\item 
	 Then the continuity equation \eqref{CEq} holds in the sense of renormalized solutions
	\begin{equation}\label{BRCEq}
\tder b(\rho)+\dvr(b(\rho)u)+(b'(\rho)\rho-b(\rho))\dvr u=0\text{ in }\mathcal D'((0,T)\times\eR^d),
\end{equation}	 
\item Moreover, for $b_\alpha$ with $\alpha\in(0,\alpha_0)$ defined as
\begin{equation}\label{BADef}
b_\alpha(s)=\begin{cases}
b(s)&s\leq\bar\rho-\alpha,\\
b(\bar\rho-\alpha)&s>\bar\rho-\alpha
\end{cases}
\end{equation}
the renormalized continuity equation holds in the form
	\begin{equation}\label{BAlphaRCEq}
	\tder b_\alpha(\rho)+\dvr(b_\alpha(\rho)u)+(b'_\alpha(\rho)\rho-b_\alpha(\rho))\dvr u=0\text{ in }\mathcal D'((0,T)\times\eR^d),
	\end{equation}
	where we set $b'_\alpha(\rho)=0$ in $\{\rho=\bar\rho-\alpha\}$.
	\end{enumerate}
	\begin{proof}
		Applying the extension procedure from \cite[Lemma 6.8]{NovStr04} we get
		\begin{equation*}
		\tder \rho+\dvr(\rho u)=0\text{ in }\mathcal D'((0,T)\times\eR^d)
		\end{equation*}
		for the extensions $\rho\in L^\infty((0,T)\times\eR^d)$ and $u\in L^2(0,T;W^{1,2}_{loc}(\eR^d))$ of $\rho$ and $u$ from assumptions of the lemma by zero in $(0,T)\times\left(\eR^d\setminus\Omega\right)$.
		Regularizing the latter equation over the spatial variables by the usual mollifier $S_\varepsilon$ with $\varepsilon>0$ yields
		\begin{equation}\label{AEIdent}
			\tder S_\varepsilon (\rho)+\dvr(S_\varepsilon(\rho) u)=r_\varepsilon(\rho, u)\text{ a.e. in }(0,T)\times\eR^d,
		\end{equation}
		where 
		\begin{equation}\label{RemCnv}
			r_\varepsilon(\rho, u)=\dvr(S_\varepsilon (\rho) u-S_\varepsilon(\rho u))\rightarrow 0\text{ in }L^r(0,T;L^r_{loc}(\eR^d))\text{ for any }r<2,
		\end{equation} cf. \cite[Lemma 6.7]{NovStr04}. We observe that \begin{equation}\label{MolRhoBound}
		\|S_\varepsilon(\rho)\|_{L^\infty(Q_T)}\leq \|\rho\|_{L^\infty(Q_T)}<\bar\rho.
		\end{equation} 
		Hence $b(S_\varepsilon(\rho))$ is well defined in $(0,T)\times\eR^d$. We multiply \eqref{AEIdent} by $b'(S_\varepsilon(\rho))$ and obtain
		\begin{equation}\label{BMolIdent}
		\begin{split}
			&\tder b(S_\varepsilon (\rho)))+\dvr(b(S_\varepsilon(\rho)) u)+\left(b'(S_\varepsilon(\rho))S_\varepsilon(\rho)-b(S_\varepsilon(\rho))\right)\dvr u\\&=b'(S_\varepsilon(\rho))r_\varepsilon(\rho, u)\text{ a.e. in }(0,T)\times\eR^d.
		\end{split}
		\end{equation}	
		The approximating property 
		\begin{equation}\label{MolRhoCnv}
		S_\varepsilon(\rho)\rightarrow\rho\text{ in }L^q(0,T;L^q_{loc}(\eR^d))\text{ for any }q\in[1,\infty)
		\end{equation}
		implies the existence of a nonrelabeled subsequence $\{S_\varepsilon (\rho)\}$ such that $S_\varepsilon(\rho)\rightarrow \rho$ a.e. in $(0,T)\times\eR^d$. Taking into account the continuity of $b^{(j)}$, where $j\in\{0,1\}$ denotes the order of the derivative, we have $b^{(j)}(S_\varepsilon(\rho))\to b^{(j)}(\rho)$ a.e. in $(0,T)\times\eR^d$. Taking into consideration also \eqref{BMon} it follows that
		\begin{equation}\label{BMolRhoBound}
		|b^{(j)}(S_\varepsilon(\rho))|\leq h^j=\begin{cases}
		\max_{[0,\bar\rho-\alpha_0]} |b^{(j)}|&\rho \leq\bar\rho-\alpha_0\\
		|b^{(j)}(\rho)|&\rho>\bar\rho-\alpha_0.
		\end{cases}
		\end{equation}
		By \eqref{BRhoInt} we infer that $h^0(\rho)$ is an integrable majorant to $b(S_\varepsilon(\rho))$ and $h^1(\rho)\overline\rho$ to $b'(S_\varepsilon(\rho))S_\varepsilon(\rho)$.
		Hence we can apply the Lebesgue dominated convergence theorem to infer
		\begin{equation}\label{BMolRhoCnv}
		\begin{alignedat}{2}
		b(S_\varepsilon(\rho))&\rightarrow b(\rho)&&\text{ in }L^2(0,T;L^2(\Sigma)),\\
		b'(S_\varepsilon(\rho))&\rightarrow b'(\rho)&&\text{ in }L^2(0,T;L^2(\Sigma)).
		\end{alignedat}
		\end{equation}
		The latter convergences and \eqref{MolRhoCnv} imply 
		\begin{equation*}
		b'(S_\varepsilon(\rho))S_\varepsilon\rho-b(S_\varepsilon(\rho))\rightarrow b'(\rho)\rho-b(\rho)\text{ in }L^2(0,T;L^2(\Sigma))
		\end{equation*}	
		for any bounded domain $\Sigma\subset\eR^d$. Combining \eqref{BMolIdent}, \eqref{RemCnv}, \eqref{BMolRhoCnv} and the latter convergence one arrives at \eqref{BRCEq} and the first assertion of the lemma is proved.		
		In order to prove the second assertion, we begin with the proof of the following auxiliary identity
		\begin{equation}\label{AuxIdent}
		(\bar\rho-\alpha)\dvr u=0\text{ a.e. in }\{\rho=\bar\rho-\alpha\}.
		\end{equation}
		To this end we consider $b\in C^1_c((0,\infty))$ such that $b(s)=s$ in $[\tfrac{3}{4}(\bar\rho-\alpha),\bar\rho-\tfrac{3}{4}\alpha]$ and $b$, $b'$ are nondecreasing in $[\bar\rho-\tfrac{\alpha}{2},\bar\rho]$. We define
		$b^+_{\alpha,\varepsilon}=S_\frac{\varepsilon}{2}(b_{\alpha+\varepsilon})$, $b^-_{\alpha,\varepsilon}=S_\frac{\varepsilon}{2}(b_{\alpha-\varepsilon})$. Then we have as $\varepsilon\to 0_+$
		\begin{equation}\label{MolBSpec}
		\begin{alignedat}{2}
		b^+_{\alpha,\varepsilon}(s),b^-_{\alpha,\varepsilon}(s)&\to b_\alpha(s)\ &&s\in[0,\bar\rho] ,\\ (b^+_{\alpha,\varepsilon})'(s),(b^-_{\alpha,\varepsilon})'(s)&\to b'_\alpha(s)\ &&s\in[0,\bar\rho]\setminus\{\bar\rho-\alpha\},\\
		(b^+_{\alpha,\varepsilon})'(\bar\rho-\alpha)&\to 0,\\
		(b^-_{\alpha,\varepsilon})'(\bar\rho-\alpha)&\to 1.
		\end{alignedat}
		\end{equation}
		By the first assertion of the lemma we have
		\begin{equation*}
		\begin{split}
		\tder b^+_{\alpha,\varepsilon}(\rho)+\dvr(b^+_{\alpha,\varepsilon}(\rho)u)+(\rho (b^+_{\alpha,\varepsilon})'(\rho)-b^+_{\alpha,\varepsilon}(\rho))\dvr u=0\text{ in }\mathcal D'((0,T)\times\eR^d),\\
		\tder b^-_{\alpha,\varepsilon}(\rho)+\dvr(b^-_{\alpha,\varepsilon}(\rho)u)+(\rho (b^-_{\alpha,\varepsilon})'(\rho)-b^-_{\alpha,\varepsilon}(\rho))\dvr u=0\text{ in }\mathcal D'((0,T)\times\eR^d).
		\end{split}
		\end{equation*}
		Letting $\varepsilon\to 0_+$ and employing the convergences from \eqref{MolBSpec} we deduce by the Lebesgue dominated convergence theorem
		from the latter identities
		\begin{equation*}
		\begin{split}
		\tder b_\alpha(\rho)+\dvr(b_{\alpha}(\rho)u)+(\rho (b_{\alpha})'(\rho)\chi_{\{\rho\neq\bar\rho-\alpha\}}-b_{\alpha}(\rho))\dvr u=0\text{ in }\mathcal D'((0,T)\times \eR^d),\\ \tder b_{\alpha}(\rho)+\dvr(b_{\alpha}(\rho)u)+(\rho (b_\alpha)'(\rho)\chi_{\{\rho\neq \bar\rho-\alpha\}}+(\bar\rho-\alpha)\chi_{\{\rho=\bar\rho-\alpha\}}-b_\alpha(\rho))\dvr u=0\text{ in }\mathcal D'((0,T)\times\eR^d).
		\end{split}
		\end{equation*}
		Subtracting the latter equations we conclude \eqref{AuxIdent}.
		
		Next, we consider for fixed $\varepsilon<\alpha$ $S_\varepsilon (b_\alpha)$, the mollification of $b_\alpha$ extended by $b(\bar\rho-\alpha)$ in $(\bar\rho,\bar\rho+1]$ and by $0$ outside of $[0,\bar\rho+1]$. As $S_\varepsilon(b_\alpha)$ is constant in a vicinity of $\overline\rho$, it fulfills \eqref{BMon} and $S_\varepsilon(b_\alpha)(\rho)$ satisfies \eqref{BRhoInt}, it follows that 
		\begin{equation}\label{MolBalphEq}
		\tder S_\varepsilon(b_\alpha(\rho))+\dvr(S_\varepsilon(b_\alpha(\rho))u)+\left(\left(S_\varepsilon(b_\alpha)\right)'(\rho)\rho-S_\varepsilon(b_\alpha(\rho))\right)\dvr u=0\text{ in }\mathcal D'((0,T)\times\eR^d).
		\end{equation}
		Furthermore, we have as $\varepsilon\to 0_+$
		\begin{equation*}
		\begin{alignedat}{2}
		S_\varepsilon(b_\alpha)(s)&\to b_\alpha(s) &&\ s\in[0,\bar\rho],\\
		(S_\varepsilon(b_\alpha))'(s)&\to b'_\alpha(s) &&\ s\in[0,\bar\rho-\alpha)\cup(\bar\rho-\alpha,\bar\rho].
		\end{alignedat}
		\end{equation*}
		Hence we infer that as $\varepsilon\to 0_+$
		\begin{equation}\label{MOlBalRhoCnv}
		\begin{alignedat}{2}
		S_\varepsilon(b_\alpha)(\rho)&\to b_\alpha(\rho) &&\text{ a.e. in }(0,T)\times\eR^d,\\
		(S_\varepsilon(b_\alpha))'(\rho)&\to b'_\alpha(\rho) &&\text{ a.e. in } (0,T)\times\eR^d\setminus\{\rho=\bar\rho-\alpha\}.
		\end{alignedat}
		\end{equation}
		Employing \eqref{AuxIdent} we conclude deduce that $\rho(S_\varepsilon(b_\alpha))'(\rho)\dvr u=0$ a.e. in $\{\rho=\bar\rho-\alpha\}$.
		Using the convergences from \eqref{MOlBalRhoCnv}, the uniform bounds with respect to $\varepsilon$ on $S_\varepsilon(b_\alpha), S_\varepsilon(b_\alpha)'$  and the Lebesgue dominated convergence theorem we pass to the limit $\varepsilon\to 0_+$ in \eqref{MolBalphEq} to conclude \eqref{BAlphaRCEq}.
	\end{proof}
\end{Lemma}

\begin{Lemma}\label{Lem:Bog}
	Let $\Omega\subset\eR^d$ be a starshaped domain with respect to a ball $B$ possessing the radius $R$. There exists a linear operator $\mathcal B: C_c^\infty(\Omega)\rightarrow C_c^\infty(\Omega)^d$ 
	such that $\dvr \mathcal B(f)=f$ provided that $\int_\Omega f=0$.
	Moreover, $\mathcal B$ can be extended in a unique way as a bounded linear operator
	\begin{enumerate}
		\item $\mathcal B: L^p(\Omega)\rightarrow W^{1,p}(\Omega)^d$ such that $\|B(f)\|_{W^{1,p}(\Omega)}\leq c\|f\|_{L^p(\Omega)}$
		\item $\mathcal B: \{f\in (W^{1,p'}(\Omega))':\langle f,1\rangle=0´\}\rightarrow L^p(\Omega)^d$ such that $\|B(f)\|_{L^p(\Omega)}\leq c\|f\|_{(W^{1,p}(\Omega))'}$
	\end{enumerate}  
 for any $p\in(1,\infty)$ where the constants $c$ take the form
 \begin{equation*}
 c=c_0(p,d)\left(\frac{\diam(\Omega)}{R}\right)^d\left(1+\frac{\diam(\Omega)}{R}\right).
 \end{equation*}
\end{Lemma}
Assertions in the following lemma are based on the results from \cite{KaLa84}
\begin{Lemma}\label{Lem:Eul}
	Let $v_0\in W^{m,2}(\eR^3)$ with $m>4$ be such that $\dvr v_0=0$ in $\eR^3$.
	Then there is $T_{max}>0$ and a classical solution $v$, unique in the class
	\begin{equation*}
	v\in C([0,T_{max}),W^{m,2}(\eR^3)^3),\ \tder v\in C([0,T_{max});W^{m-1,2}(\eR^3)^3)
	\end{equation*}
	to the initial value problem
	\begin{align*}
	\tder v+v\cdot\nabla v+\nabla \Pi=0&\text{ in }(0,T_{max})\times\eR^3,\\
	v(0,\cdot)=v_0, \dvr v_0=0&\text{ in }\eR^3.
	\end{align*}
	Furthermore, the associate pressure $\Pi$ can be expressed as
	\begin{equation*}
	\Pi=(-\Delta)^{-1}\dvr\dvr(v\otimes v).
	\end{equation*}
	implying particularly that $\Pi\in C^1([0,T];C^1(\eR^3)\cap W^{1,2}(\eR^3))$, $T\in (0,T_{max})$.
\end{Lemma}

\section*{Acknowledgment}
{\it \v S. N. and M. K. have been supported by the Czech Science Foundation (GA\v CR) project 22-01591S.  Moreover, \it \v S. N.  and M. K. have been supported by  Praemium Academiæ of \v S. Ne\v casov\' a. The Institute of Mathematics, CAS is supported by RVO:67985840.}

\section*{Conflict of interest}
On behalf of authors, the corresponding author states that there is no  conflict of interest. 

\bibliographystyle{plain}

\end{document}